\documentclass{amsart}
\usepackage{tikz}
\usepackage{booktabs}
\usepackage{graphicx}
\usepackage{hyperref}
\usepackage{xparse}

\usepackage[utf8x]{inputenc}
\usepackage[T2A]{fontenc}

\usepackage{amsmath,amsthm,amsopn,amstext,amscd,amsfonts,amssymb,mathrsfs,mathtools}
\usepackage{dsfont}
\usepackage{comment}
\usepackage{xcolor}
\usepackage{float}
\usepackage[active]{srcltx}
\usepackage{graphicx, mathdots}

\usepackage{mathtools}

\newtheorem{problem}{\sc Problem}[section]
\newtheorem{definition}{\sc Definition}[section]
\newtheorem{theorem}{\sc Theorem}[section]

\newtheorem{lemma}{\sc Lemma }[section]
\newtheorem{eje}{\sc Example }[section]
\newtheorem{coro}{\sc Corollary}[section]

\usetikzlibrary{decorations.pathreplacing}
\usepackage{xcolor}

\usepackage{listings}
\usepackage{xcolor}

\definecolor{codebg}{RGB}{248,248,248}
\definecolor{codeframe}{RGB}{220,220,220}
\definecolor{codenumbers}{RGB}{120,120,120}
\definecolor{pykeyword}{RGB}{0,92,184}
\definecolor{pystring}{RGB}{163,21,21}
\definecolor{pycomment}{RGB}{87,87,87}

\usepackage{listings}
\usepackage{xcolor}

\definecolor{codebg}{RGB}{248,248,248}
\definecolor{codeframe}{RGB}{220,220,220}
\definecolor{codenumbers}{RGB}{120,120,120}

\definecolor{wlkeyword}{RGB}{0,92,184}   
\definecolor{wlstring}{RGB}{163,21,21}   
\definecolor{wlcomment}{RGB}{90,90,90}   

\lstdefinelanguage{Wolfram}{
  morekeywords={
    ClearAll,Clear,Module,Block,With,Set,SetDelayed,Rule,RuleDelayed,
    Table,Do,While,If,Which,For,Return,Function,Map,MapIndexed,MapThread,
    Apply,ReplaceAll,ReplaceRepeated,DeleteCases,Select,Sort,SortBy,Tuples,
    ConstantArray,Length,Append,Join,Total,Sum,Product,Exp,Sin,Cos,Conjugate,
    Simplify,FullSimplify,Expand,Coefficient,Exponent,Solve,NSolve,N,Chop,Print
  },
  sensitive=true,
  morecomment=[l]{(*},
  morecomment=[s]{(*}{*)},
  morestring=[b]",
}

\lstdefinestyle{wlcode}{
  language=Wolfram,
  basicstyle=\ttfamily\small,
  backgroundcolor=\color{codebg},
  frame=single,
  rulecolor=\color{codeframe},
  framerule=0.4pt,
  columns=fullflexible,
  keepspaces=true,
  showstringspaces=false,
  breaklines=true,
  breakatwhitespace=false,
  tabsize=2,
  numbers=left,
  numberstyle=\scriptsize\color{codenumbers},
  numbersep=8pt,
  xleftmargin=2.2em,
  framexleftmargin=1.6em,
  keywordstyle=\color{wlkeyword}\bfseries,
  commentstyle=\color{wlcomment}\itshape,
  stringstyle=\color{wlstring},
}

\usepackage{algorithm}
\usepackage{algpseudocode}
\usepackage[most]{tcolorbox}

\tcbset{
  algobox/.style={
    enhanced,
    colback=black!3,       
    colframe=black!15,     
    boxrule=0.4pt,
    arc=2mm,
    left=2mm,right=2mm,top=1.5mm,bottom=1.5mm,
  }
}

\usepackage[active]{srcltx}

\usepackage{graphicx, epsfig, subfig}
\usepackage{lscape}
\usepackage{rotating}
\usepackage{caption}
\usepackage{multicol}
\usepackage{colortbl}
\usepackage{nicematrix}

\begin{document}

\title[Wendroff's Theorem beyond consecutive degrees]{Wendroff's Theorem Beyond Consecutive Degrees and Related Inverse Spectral Problems}

\author{K. Castillo}
\address{CMUC, Department of Mathematics, University of Coimbra, 3000-143 Coimbra, Portugal}
\email{kenier@mat.uc.pt}

\author{G. Gordillo-N\'u\~nez}
\address{CMUC, Department of Mathematics, University of Coimbra, 3000-143 Coimbra, Portugal}
\email{up202310693@up.pt}

\date{\today}

\begin{abstract}
{
A classical theorem of Wendroff shows that one may reconstructs a sequence of orthogonal polynomials on the real line from two non-constant polynomials of consecutive degrees whose zeros strictly interlace on the real line. In this note we extend this result to arbitrary non-constant polynomials. The reconstruction may be formulated via a Vandermonde-type linear system and recast as an underdetermined inverse spectral problem, in which the spectra of a finite Jacobi matrix and of one of its leading principal submatrices are prescribed. In addition, the analogous result on the unit circle is established by reconstructing a sequence of paraorthogonal polynomials from two arbitrary non-constant polynomials whose zeros strictly interlace on the unit circle. In this setting, the Jacobi matrix is replaced by a finite unitary pentadiagonal matrix, and the spectral data consist of the spectrum of the full matrix together with that of a rank-one perturbation of a leading principal submatrix. Strict interlacing of zeros is shown to be a necessary and sufficient condition for solvability, and explicit constructions of the associated polynomial families and matrices are provided. Finally, an algorithm and several illustrative examples are presented.
}
\end{abstract}

\maketitle
\section{Introduction}\label{sec:intro}
In 1961, Wendroff published a short note \cite{Wendroff1961} motivated by the following question: {  Given \(n\) distinct real numbers \(\{x_k\}_{k=1}^n \subset [a,b]\), can they be realised as the zeros of the degree-\(n\) polynomial of an orthogonal polynomial sequence on \([a,b]\)?} Wendroff proved that the answer is affirmative, and in fact established a stronger statement. Namely, if one also prescribes $n-1$ distinct real numbers $\{y_k\}_{k=1}^{n-1}$ that strictly interlace the $\{x_k\}_{k=1}^n$ in the sense that
\[
x_i<y_i<x_{i+1}, \quad i=1,\dots,n-1,
\]
then there exists a { (finite) sequence $\{P_k\}_{k=0}^{n}$ of monic polynomials
orthogonal with respect to some positive Borel measure $\mu$ supported on $[a,b]$, i.e.
\[
\int P_j(x)\,P_k(x)\,d\mu(x)=0,\quad j\neq k,
\]
such that}
\[
P_n(x)=(x-x_1)\cdots(x-x_n), \quad
P_{n-1}(x)=(x-y_1)\cdots(x-y_{n-1}).
\]
{ Of course, this finite orthogonal family admits infinitely many extensions to a full sequence of orthogonal polynomials on the real line (OPRL): one may continue the corresponding three-term recurrence beyond degree $n$ in infinitely many ways, and Favard's theorem (see, for instance, \cite[Theorem~2.9]{CastilloPetronilho}) then produces an OPRL for each such continuation. Wendroff's result had already appeared, albeit without proof and only in passing, in a footnote of Geronimus from 1946 \cite[p.~744]{Geronimus1946}. We follow standard usage in referring to this result as \emph{Wendroff’s theorem}; see, for instance, \cite[Remark~3, p.~999]{MartinezFinkelshteinSimanekSimon2019}.}
From a modern viewpoint, Wendroff's theorem fits naturally into inverse spectral theory for finite Jacobi matrices. Indeed, if $\{P_k\}_{k\ge0}$ is a monic OPRL family, then it satisfies the three-term recurrence
\[
P_{-1}(x)=0,\quad P_0(x)=1,\quad
P_{k+1}(x)=(x-\beta_k)P_k(x)-\gamma_k P_{k-1}(x),\quad k\ge0.
\]
{  In particular, one has \(\beta_k\in\mathbb{R}\) and \(\gamma_k>0\) for \(k\ge 1\).} From these coefficients one forms the $n\times n$ Jacobi matrix
\begin{equation}
\label{eq:jacobi-matrix}
    J_n =
\begin{bmatrix}
\;\beta_0   & 1        &          &          &          &        \\[3pt]
\;\gamma_1  & \beta_1  & 1        &          &          &        \\[3pt]
          & \gamma_2 & \beta_2  & 1   &          &        \\[3pt]
          &          & \ddots   & \ddots   & \ddots   &        \\[3pt]
          &          &          & \gamma_{n-2} & \beta_{n-2} & 1\; \\[3pt]
          &          &          &          & \gamma_{n-1} & \beta_{n-1}\;
\end{bmatrix}.
\end{equation}
For $m\le n$ we denote by $J_m$ its leading principal submatrix of order $m$. If $\Delta_k(x)=\det(xI_k-J_k)$ with $\Delta_0(x)=1$, then a cofactor expansion shows that $\{\Delta_k\}_{k\ge0}$ satisfies the same recurrence as $\{P_k\}_{k\ge0}$ with the same initial conditions (see, e.g., \cite[Remark~3.4]{CastilloPetronilho}). Hence, for each $k\ge1$,
\[
P_k(x)=\det(xI_k-J_k),
\]
or, equivalently,
\[
\mathrm{zeros}(P_k)=\sigma(J_k).
\]
Consequently, Wendroff's question may be viewed as an inverse eigenvalue problem: can one construct a Jacobi matrix $J_n$ whose spectrum is the prescribed set $\{x_k\}_{k=1}^n$? Moreover, his theorem is precisely an existence statement for a two-spectra configuration: prescribing two consecutive spectra $Z_n$ and $Z_{n-1}$ (equivalently, the zeros of $P_n$ and $P_{n-1}$), strict interlacing guarantees the existence of a Jacobi matrix $J_n$ such that
\[
\sigma(J_n)=Z_n, \quad \sigma(J_{n-1})=Z_{n-1}.
\]

Independently (and after Wendroff's note), Hochstadt initiated the finite two-spectra theory for Jacobi matrices in 1967. In \cite{Hochstadt1967} he proved an algebraic uniqueness result in this setting, and in \cite{Hochstadt1974} he established existence and derived an explicit reconstruction procedure, formulated in terms of trailing principal submatrices. This is essentially the same configuration as Wendroff's once one observes that trailing and leading blocks are equivalent up to index reversal. Thus, Wendroff's and Hochstadt's results may be regarded as complementary facets of the same principle connecting OPRL and Jacobi matrices.

After Hochstadt, inverse eigenvalue problems for Jacobi matrices have attracted sustained attention, driven both by the intrinsic structure of tridiagonal models and by applications (notably vibration theory and other one-dimensional discrete models in mathematical physics). A common theme is to prescribe enough data to make the reconstruction determined (or overdetermined), matching the $2n-1$ degrees of freedom of an $n\times n$ Jacobi matrix. For instance, the interlacing structure of spectra of several principal submatrices is developed in \cite{HigginsJohnson2016}, and a recent ``mixed data'' determined problem (combining spectral and structural information) is treated in \cite{YangWuZhang2023}; see also the application-driven accounts \cite{Gladwell1986,Kay2010,VinetZhedanov2012}.

The present paper takes a different route. { The aim is to extend Wendroff’s theorem beyond consecutive degrees, and indeed much further, by establishing analogous results for polynomials with zeros on the unit circle.} Fix $1 \le m<n$ and two sets of distinct real numbers
\[
Z_n=\{x_k\}_{k=1}^{n}, \quad Z_m=\{y_k\}_{k=1}^{m}.
\]
We ask when there exists a finite OPRL family $\{P_k\}_{k=0}^{n}$ such that
\[
\mathrm{zeros}(P_n)=Z_n,\quad
\mathrm{zeros}(P_m)=Z_m.
\]
Equivalently, in Jacobi--matrix language, we prescribe the spectra of $J_n$ and of its leading principal submatrix $J_m$. When $m=n-1$ this is the classical two-spectra regime (Wendroff--Hochstadt), where the data are determined and the solution is unique. For $m<n-1$, however, one prescribes only $m+n$ real numbers, so the inverse problem is genuinely underdetermined and uniqueness cannot be expected. More precisely, we consider the following problem.

\begin{problem}[{Underdetermined} two-spectra problem for Jacobi matrices]\label{prob:underdetermined-oprl}
Fix integers $1\le m<n$ and distinct real numbers
\[
Z_n=\{x_k\}_{k=1}^{n}, \quad Z_m=\{y_k\}_{k=1}^{m}.
\]
Determine necessary and sufficient conditions on their mutual position under which there exists a Jacobi matrix $J_n$ such that, with $J_m$ its leading principal submatrix of order $m$,
\[
\sigma(J_n)=Z_n,\quad \sigma(J_m)=Z_m.
\]
\end{problem}

By the standard eigenvalue interlacing for principal submatrices, an appropriate strict interlacing between $\{x_k\}_{k=1}^{n}$ and $\{y_k\}_{k=1}^{m}$ is necessary. Our main result shows that, even in the underdetermined regime $m<n-1$, this condition is also sufficient. We use the following definition of strict interlacing on $\mathbb{R}$.

\begin{definition}[Strict interlacing on $\mathbb{R}$]\label{def:strict-interlacing-oprl}
Let $x_1<\cdots<x_n$ and $y_1<\cdots<y_m$ with $m<n$. We say that
$\{y_k\}_{k=1}^m$ strictly interlaces $\{x_j\}_{j=1}^n$ if there exist so-called interlacing indices
\[
i_0=0<i_1<\cdots<i_m<n=i_{m+1}
\]
such that $y_k \in (x_{i_k},x_{i_k+1})$, for $k=1,\dots,m$. 
Equivalently, each open interval $(x_i,x_{i+1})$ contains at most one point of $\{y_k\}_{k=1}^m$.
\end{definition}

In particular, this yields an extension of Wendroff's theorem beyond consecutive degrees: if two sets of simple real zeros $\{x_j\}_{j=1}^n$ and $\{y_k\}_{k=1}^m$ satisfy strict interlacing in the above sense, then there exists an OPRL family whose degree-$n$ and degree-$m$ members have precisely these zeros (equivalently, there exists a Jacobi pair $J_m\subset J_n$ with the prescribed spectra).

We also consider the analogous question on the unit circle. Let
\[
\mathbb{D}=\{z\in\mathbb{C}:|z|<1\},\quad
\mathbb{S}^1=\{z\in\mathbb{C}:|z|=1\}.
\]
{ 
{ Given a positive Borel measure $\nu$ on $\mathbb S^1$, let
$\{\Phi_k\}_{k\ge0}$ denote the corresponding sequence of monic polynomials
orthogonal with respect to $\nu$ (OPUC), i.e.
\[
\int \Phi_j(z)\,\overline{\Phi_k(z)}\,d\nu(z)=0,\quad j\neq k.
\]
For each $k\ge0$, define the reversed polynomials
\[
\Phi_k^*(z)=z^k\,\overline{\Phi_k(1/\overline z)}.
\]
The sequence $\{\Phi_k\}_{k\ge0}$ is uniquely encoded by the Verblunsky coefficients
$\alpha_k\in\mathbb D$ through Szeg\H{o}'s recurrence
\[
\Phi_{k+1}(z)=z\Phi_k(z)-\overline{\alpha_k}\,\Phi_k^*(z),
\quad
\Phi_{k+1}^*(z)=\Phi_k^*(z)-\alpha_k\,z\,\Phi_k(z),
\quad k\ge0,
\]
(where, equivalently, $\alpha_k=-\overline{\Phi_{k+1}(0)}$).}
Fix $n\ge1$ and a boundary parameter $b_n\in\mathbb{S}^1$, and form the so-called
degree-$n$ paraorthogonal polynomial on the unit circle (POPUC)
\begin{equation}
\label{eq:popuc}
\Psi_n(z)=z\Phi_{n-1}(z)-\overline{b_n}\,\Phi_{n-1}^*(z).
\end{equation}
Then $\Psi_n$ may be viewed as the unit-circle analogue of the real-line object in Wendroff's theorem:
its zeros are simple and lie on $\mathbb S^1$, and varying the boundary parameter $b_n\in\mathbb S^1$
induces a separation (indeed, an interlacing) of these zeros along $\mathbb S^1$; see \cite[Ch.~2]{S05I}. This leads naturally to a Wendroff-type inverse problem on $\mathbb S^1$:
given a finite set of points on $\mathbb S^1$, when can it be realised as the zero set of a POPUC?
More generally, given two prescribed finite sets on $\mathbb S^1$, when do there exist Verblunsky
coefficients $\{\alpha_k\}_{k\ge0}\subset\mathbb D$ and boundary parameters $b_n,b_m\in\mathbb S^1$
such that the corresponding POPUC of degrees $n$ and $m$ have exactly those sets as their zero sets?
}

The finite-dimensional spectral model underlying $\Psi_n$ is a unitary pentadiagonal matrix, which we denote by
\[
\mathcal{C}_n=\mathcal{C}(\alpha_0,\dots,\alpha_{n-1},b_n),
\]
constructed canonically from $\alpha_0,\dots,\alpha_{n-1}\in\mathbb{D}$ and $b_n\in\mathbb{S}^1$. { The matrix $\mathcal{C}_n$,  which
is unitarily similar to the Schur parametric form of an upper Hessenberg matrix with
positive subdiagonal elements \cite{BohnhorstBunseGerstnerFassbender2000}, was presented by Bunse-Gerstner and Elsner \cite{BE91} (see \cite[Section $12.2.10$]{GV83} and \cite[Definition $3.3$ and Lemma $3.4$]{B93}).  The explicit unitary pentadiagonal or double-staircase form of $\mathcal{C}_n$ (referred as {\em Doppel-Treppen-Matrix} in the original German source) was studied extensively by Bohnhorst \cite{B93}, see \cite[Equation $3.9$]{B93} and \cite[Figure $1.1$]{KN07} for an $8$-by-$8$ example. The matrix $\mathcal{C}_n$ becomes a very popular object in the Mathematical Physics and Orthogonal Polynomials communities after the work \cite{CanteroMoralVelazquez2003}, specially after Simon's monographs \cite{S05I, S05II}  where it was called (improperly) CMV matrix (see \cite{S07b}).}
 Its characteristic polynomial is, up to a unimodular constant, $\Psi_n$, and hence
\[
\mathrm{zeros}(\Psi_n)=\sigma(\mathcal{C}_n).
\]
Accordingly, the above Wendroff-type questions on $\mathbb{S}^1$ are equivalent to inverse spectral problems for unitary pentadiagonal matrices.

{  General interlacing theorems for these matrices and their leading principal submatrices of arbitrary order, in connection with rank-one perturbations, may be found in \cite{CP18}.} Converse theorems in this direction have been obtained in several forms. A concrete consecutive-degree analogue appears in \cite{CastilloCruzPerdomo2016}: under the appropriate strict interlacing condition on $\mathbb{S}^1$ (allowing at most one common zero), two consecutive POPUC arise from a unique choice of parameters; see \cite[Theorem~4.1]{CastilloCruzPerdomo2016}. A broader viewpoint, including converse statements for Blaschke products and POPUC together with parameter counting in the unitary setting, is developed in \cite{MartinezFinkelshteinSimanekSimon2019}.

Here we again move beyond consecutive degrees. Fix integers $1\le m<n$ and two finite sets of distinct points on $\mathbb{S}^1$,
\[
Z_n=\{\xi_k\}^n_{k=0}\subset\mathbb{S}^1,\quad
Z_m=\{\zeta_k\}^m_{k=0}\subset\mathbb{S}^1.
\]
We ask for conditions ensuring that $Z_n$ and $Z_m$ can be realized as spectra of a compatible pair of unitary pentadiagonal matrices of orders $n$ and $m$ (equivalently, as zero sets of POPUC of degrees $n$ and $m$). As on the real line, when $m=n-1$ one is in the consecutive (determined) regime, while for $m<n-1$ the problem is underdetermined, in contrast with the determined inverse unitary
pentadiagonal problems treated in, e.g., \cite{ArlinskiiGolinskiiTsekanovskii2008,GolinskiiKudryavtsev2009, GolinskiiKudryavtsev2010}. We note that \cite{GolinskiiKudryavtsev2010} also considers underdetermined inverse problems of a different nature; however, in contrast with the present work, those settings are not given a complete solvability characterization in terms of necessary and sufficient conditions. More precisely, we consider the following problem on $\mathbb{S}^1$.

\begin{problem}[{Underdetermined} two-spectra problem for unitary pentadiagonal matrices]\label{prob:underdetermined-POPUC}
Fix integers $1\le m<n$ and two sets
\[
Z_n=\{\xi_k\}^n_{k=0}\subset\mathbb{S}^1,\quad
Z_m=\{\zeta_k\}^m_{k=0}\subset\mathbb{S}^1.
\]
with all points distinct. Determine necessary and sufficient conditions on their relative position along $\mathbb{S}^1$ under which there exist parameters
\(
\alpha_0,\dots,\alpha_{n-1}\in\mathbb{D},\)
and \(b_n,b_m\in\mathbb{S}^1,\) such that the associated unitary pentadiagonal matrices
\[
\mathcal{C}_n=\mathcal{C}(\alpha_0,\dots,\alpha_{n-1},b_n),\quad
\mathcal{C}_m=\mathcal{C}(\alpha_0,\dots,\alpha_{m-1},b_m)
\]
satisfy
\[
\sigma(\mathcal{C}_n)=Z_n,\quad \sigma(\mathcal{C}_m)=Z_m.
\]
\end{problem}

As in the real-line setting, strict interlacing is necessary by the standard separation of POPUC zeros (equivalently, of eigenvalues under boundary variation). {  The main result shows that, in the nonconsecutive regime \(m<n-1\), this condition is also sufficient. Consequently, Wendroff's theorem is extended beyond consecutive degrees on the unit circle. In particular, the following definition of strict interlacing on the unit circle will be used.}

\begin{definition}[Strict interlacing on $\mathbb{S}^1$]\label{def:strict-interlacing-POPUC}
Let
\[
Z_n=\{\zeta_j\}_{j=1}^n \subset \mathbb{S}^1,\quad
Z_m=\{\xi_k\}_{k=1}^m \subset \mathbb{S}^1,
\]
be sets of pairwise distinct points with $Z_n\cap Z_m=\emptyset$. We say that $Z_m$ \emph{strictly interlaces} $Z_n$ if, for every $k=1,\dots,m$, the open
counterclockwise arc from $\zeta_k$ to $\zeta_{k+1}$ contains at most one point of $Z_m$, where we set
$\zeta_{m+1}=\zeta_1$.
\end{definition}

{  This paper is organized as follows. Section~\ref{sec:prelim} records the technical ingredients used throughout the paper and, in particular, formulates the real-line and unit-circle reconstruction problems in terms of quadrature representations with strictly positive weights. Section~\ref{sec:OPRL} addresses the real-line setting, establishing that strict interlacing is the sharp solvability criterion for the underdetermined two-spectra problem for Jacobi matrices and describing the corresponding solution set. Section~\ref{sec:POPUC} proves the analogous statement on the unit circle for paraorthogonal polynomials and unitary pentadiagonal matrices. Section~\ref{sec:algos} presents a reconstruction algorithm together with illustrative examples. Finally, Section~\ref{sec:remarks} offers concluding remarks and discusses extensions beyond the two-spectra regime.}

\section{Preliminary results}\label{sec:prelim}
{  We record a few technical ingredients that will be used repeatedly later on. As agreed in the
Introduction, we remain throughout in the positive definite setting and interpret orthogonality
via positive Borel measures. For background on OPRL and Gauss--Jacobi quadrature we refer to
\cite[Chapter~3]{CastilloPetronilho}, whereas for OPUC/POPUC and Szeg\H{o} quadrature see
\cite{S05I,S05II}.} In view of Section~\ref{sec:intro}, zeros of the relevant polynomials already admit a spectral
interpretation (Jacobi matrices on \(\mathbb R\) and the unitary pentadiagonal matrices
on \(\mathbb S^1\)). The main point here is a complementary interpretation: zeros as
nodes of quadrature formulas with strictly positive weights. This is the mechanism behind
Lemmas~\ref{lemma:OPRLlemma} and~\ref{lemma:POPUClemma}.

\medskip
\noindent\textbf{Gauss--Jacobi quadrature (on $\mathbb R$).}
Let \(\mu\) be a positive Borel measure on \(\mathbb R\) and let \(\{P_k(\cdot;\mu)\}_{k\ge 0}\) be its
monic OPRL. Thus, \(P_m(\cdot;\mu)\) satisfies
\[
\int P_m(x;\mu)\,x^k\,d\mu(x)=0,\quad k=0,1,\dots,m-1.
\]
If \(x_{n,1},\dots,x_{n,n}\) are the zeros of \(P_n(\cdot;\mu)\), then there exist weights
\(\omega_{n,1},\dots,\omega_{n,n}>0\) such that
\begin{equation}\label{eq:gauss-jacobi}
\int p(x)\,d\mu(x)=\sum_{j=1}^{n}\omega_{n,j}\,p(x_{n,j}),
\quad \deg p\le 2n-1,
\end{equation}
see \cite[Theorem~3.4]{CastilloPetronilho}. In particular, the nodes are precisely the zeros of
\(P_n\), and the weights are strictly positive.

\begin{lemma}\label{lemma:OPRLlemma}
Fix integers \(1\le m<n\) and two sets of distinct real numbers
\[
Z_n=\{x_k\}_{k=1}^{n}, \quad Z_m=\{y_k\}_{k=1}^{m}.
\]
Let
\[
P_n(x)=(x-x_1)\cdots(x-x_n),\quad P_m(x)=(x-y_1)\cdots(x-y_m).
\]
Then the following statements are equivalent:
\begin{enumerate}
\item[(i)] There exists a positive Borel measure \(\mu\) on \(\mathbb R\) whose monic OPRL satisfy
\(P_n(\cdot;\mu)=P_n\) and \(P_m(\cdot;\mu)=P_m\).
\item[(ii)] There exists a Jacobi matrix \(J_n\) \eqref{eq:jacobi-matrix} such that, if \(J_m\) is its
leading principal submatrix of order \(m\),
\[
\sigma(J_n)=Z_n,\quad \sigma(J_m)=Z_m.
\]
\item[(iii)] There exist weights \(\omega_1,\dots,\omega_n>0\) such that
\begin{equation}\label{eq:vandermonde}
\sum_{j=1}^{n}\omega_j\,P_m(x_j)\,x_j^k=0,\quad k=0,1,\dots,m-1.
\end{equation}
\end{enumerate}
\end{lemma}

\begin{proof}
The equivalence (i)\(\Leftrightarrow\)(ii) is the standard OPRL/Jacobi correspondence recalled in
Section~\ref{sec:intro}; see also \cite[Remark~3.4]{CastilloPetronilho}.

(i)\(\Rightarrow\)(iii).
Assume (i). Apply Gauss--Jacobi quadrature \eqref{eq:gauss-jacobi} with \(n\) to the zeros
\(x_1,\dots,x_n\) of \(P_n\), obtaining weights \(\omega_1,\dots,\omega_n>0\). Since
\(P_m(\cdot;\mu)=P_m\), we have \(\int P_m(x)x^k\,d\mu(x)=0\) for \(k=0,\dots,m-1\). As
\(\deg(P_m x^k)\le 2m-1\le 2n-1\), \eqref{eq:gauss-jacobi} yields \eqref{eq:vandermonde}.

(iii)\(\Rightarrow\)(i).
Assume (iii). Consider an arbitrary positive definite Borel measure $\mu$ such that its first $2n-1$ moments coincide with those given by \(\mu_0=\sum_{j=1}^{n}\omega_j\,\delta_{x_j}\), where \(\delta_{x_0}\) denotes the
point mass at \(x_0\). Since \(\omega_j>0\), the Gram--Schmidt process applied to
\(1,x,\dots,x^n\) with respect to $\mu$ produces a finite monic orthogonal family
\(\{Q_k(\cdot;\mu)\}_{k=0}^{n}\). By construction, \(Q_n\) is monic and must coincide with $P_n$. Moreover, \eqref{eq:vandermonde} says \(P_m\) is orthogonal to
\(1,x,\dots,x^{m-1}\) with respect to $\mu$, so \(Q_m=P_m\) by uniqueness of the monic orthogonal polynomial
of degree \(m\).
\end{proof}

\noindent
Lemma~\ref{lemma:OPRLlemma} is the basic device used in Section~\ref{sec:OPRL} for solving
Problem~\ref{prob:underdetermined-oprl}: under strict interlacing, we characterize the solution set of
\eqref{eq:vandermonde} under \(\omega_j>0\), and hence all measures \(\mu\) (up to moments of order
\(2n-1\)) and all Jacobi matrices \(J_n\) realizing the prescribed data.

\medskip
\noindent\textbf{Szeg\H{o} quadrature (on $\mathbb S^1$).}
Let \(\nu\) be a positive Borel measure on \(\mathbb S^1\), and let
\(\{\Phi_k(\cdot;\nu)\}_{k\ge 0}\) be the associated monic OPUC with Verblunsky coefficients
\(\alpha_k\in\mathbb D\). The orthogonality takes the form
\[
\int \Phi_m(z;\nu)\,z^{-k}\,d\nu(z)=0,\quad k=0,1,\dots,m-1.
\]
Fix \(n\ge 1\) and \(b_n\in\mathbb S^1\), and define the POPUC \(\Psi_n(\cdot;\nu,b_n)\) via
\eqref{eq:popuc}. Its zeros are simple and lie on \(\mathbb S^1\), and they serve as nodes of
Szeg\H{o} quadrature: if \(\mathrm{zeros}(\Psi_n)=\{\zeta_{n,j}\}_{j=1}^n\), then there exist unique
weights \(\omega_{n,1},\dots,\omega_{n,n}>0\) such that
\begin{equation}\label{eq:szego-quad}
\int p(z)\,d\nu(z)=\sum_{j=1}^{n}\omega_{n,j}\,p(\zeta_{n,j}),
\quad
p\in\mathrm{span}\{z^{-(n-1)},\dots,z^{n-1}\},
\end{equation}
see, for example, \cite{Gragg1993}.

\begin{lemma}\label{lemma:POPUClemma}
Fix integers \(1\le m<n\) and two sets
\[
Z_n=\{\zeta_j\}_{j=1}^{n}\subset\mathbb S^1,\quad
Z_m=\{\xi_k\}_{k=1}^{m}\subset\mathbb S^1,
\]
with \(Z_n\cap Z_m=\emptyset\). Define the monic polynomials
\[
\Psi_n(z)=(z-\zeta_1)\cdots(z-\zeta_n),\quad
\Psi_m(z)=(z-\xi_1)\cdots(z-\xi_m),
\]
and set \(b_n=-\overline{\Psi_n(0)}\) and \(b_m=-\overline{\Psi_m(0)}\).
Then the following statements are equivalent:
\begin{enumerate}
\item[(i)] There exists a positive Borel measure \(\nu\) on \(\mathbb S^1\) such that
\[
\Psi_n(z)=\Psi_n(z;\nu,b_n),\quad
\Psi_m(z)=\Psi_m(z;\nu,b_m).
\]
\item[(ii)] There exist Verblunsky coefficients \(\alpha_0,\dots,\alpha_{n-2}\in\mathbb D\) such that,
with the unitary pentadiagonal matrices \(\mathcal C\) introduced in Section~\ref{sec:intro},
\[
\sigma\bigl(\mathcal{C}(\alpha_0,\dots,\alpha_{n-2},b_n)\bigr)=Z_n,\quad
\sigma\bigl(\mathcal{C}(\alpha_0,\dots,\alpha_{m-2},b_m)\bigr)=Z_m.
\]
\item[(iii)] There exist weights \(\omega_1,\dots,\omega_n>0\) such that
\begin{equation}\label{eq:vandermondePOPUC}
\sum_{j=1}^{n}\omega_j\,\Psi_m(\zeta_j)\,\zeta_j^{-k}=0,\quad k=1,\dots,m-1.
\end{equation}
\end{enumerate}
\end{lemma}

\begin{proof}
The equivalence (i)\(\Leftrightarrow\)(ii) is the standard unitary
pentadiagonal/POPUC correspondence recalled in
Section~\ref{sec:intro}; see \cite{CanteroMoralVelazquez2003,S05I,S07b}.

(i)\(\Rightarrow\)(iii).
Assume (i). Apply Szeg\H{o} quadrature \eqref{eq:szego-quad} with \(n\) to \(\Psi_n(\cdot;\nu,b_n)\),
whose zeros are \(\zeta_1,\dots,\zeta_n\), obtaining weights \(\omega_1,\dots,\omega_n>0\). Since
\(\Psi_m=\Psi_m(\cdot;\nu,b_m)\), the standard paraorthogonality relations give
\[
\int \Psi_m(z)\,z^{-k}\,d\nu(z)=0,\quad k=1,\dots,m-1,
\]
and each Laurent polynomial \(\Psi_m(z)z^{-k}\) lies in
\(\mathrm{span}\{z^{-(n-1)},\dots,z^{n-1}\}\) because \(m<n\). Therefore, \eqref{eq:szego-quad} yields
\eqref{eq:vandermondePOPUC}.

(iii)\(\Rightarrow\)(i).
Assume (iii). Consider an arbitrary positive definite measure $\nu$ on the unit circle $\mathbb{S}^1$ such that its moments from $k=-(n-1),\dots,n-1$ coincide with those defined by the discrete measure
\[
\nu_0=\sum_{j=1}^{n}\omega_j\,\delta_{\zeta_j}.
\]
Since \(\omega_j>0\), Gram--Schmidt applied to \(1,z,\dots,z^{n-1}\) with respect to $\nu$ produces the
monic OPUC \(\{\Phi_k(\cdot;\nu)\}_{k=0}^{n-1}\) and their reversals \(\{\Phi_k^*(\cdot;\nu)\}\). Because \(\nu\) is supported on \(Z_n=\{\zeta_j\}_{j=1}^n\), the monic polynomial
\(\Psi_n(z)\) is paraorthogonal of degree \(n\) with respect to \(\nu\)
(it vanishes on the support, hence is orthogonal to all \(z^{-k}\) for \(k=1,\dots,n-1\)). Likewise,
(iii) is precisely the statement that \(\Psi_m(z)\) is paraorthogonal of
degree \(m\) with respect to the same \(\nu\).

Fix \(\ell\in\{m,n\}\). The degree-\(\ell\) paraorthogonal polynomials for \(\nu\) form a
two-dimensional vector space, spanned by \(z\,\Phi_{\ell-1}(z;\nu)\) and \(\Phi_{\ell-1}^*(z;\nu)\);
see \cite[Chapter~2]{S05I}. Therefore \(\Psi_\ell\) must be a linear combination of these two
polynomials. Since \(\Psi_\ell\) is monic, the coefficient of \(z\,\Phi_{\ell-1}\) must be \(1\).
Evaluating at \(z=0\), the term \(z\,\Phi_{\ell-1}(z;\nu)\) vanishes, while
\(\Phi_{\ell-1}^*(0;\nu)\neq 0\); hence the remaining coefficient is uniquely determined by the value
\(\Psi_\ell(0)\). With the choice \(b_\ell=-\overline{\Psi_\ell(0)}\), this gives
\[
\Psi_\ell(z)=z\,\Phi_{\ell-1}(z;\nu)-\overline{b_\ell}\,\Phi_{\ell-1}^*(z;\nu),
\quad \ell\in\{m,n\},
\]
i.e.,
\[
\Psi_n(z)=\Psi_n(z;\nu,b_n),\quad \Psi_m(z)=\Psi_m(z;\nu,b_m).
\]
This gives (i) and completes the proof.
\end{proof}

\noindent
Lemma~\ref{lemma:POPUClemma} reduces the unit-circle problem to the analysis of the solution set of
\eqref{eq:vandermondePOPUC} under the constraint \(\omega_j>0\). Under strict interlacing on
\(\mathbb S^1\), Section~\ref{sec:POPUC} solves Problem \ref{prob:underdetermined-POPUC} and describes the
corresponding family of unitary pentadiagonal matrices \(\mathcal{C}\) realizing the prescribed
spectra.

\section{Main results on $\mathbb R$}\label{sec:OPRL}
We treat first the real-line problem. Fix integers $1\le m<n$ and two sets of ordered and real distinct points
\[
Z_n=\{x_j\}_{j=1}^n,\quad Z_m=\{y_k\}_{k=1}^m.
\]
As explained in Section~\ref{sec:prelim}, solvability is equivalent to the
existence of weights $\omega_1,\dots,\omega_n>0$ satisfying the Vandermonde-type system
\eqref{eq:vandermonde}. Accordingly, we introduce the set of strictly positive solutions
\[
\mathcal{W}_{>0}(Z_n,Z_m)
=
\Bigl\{ \boldsymbol{\omega}=(\omega_1,\dots,\omega_n)\in\mathbb{R}^n:
\mathbf{A}\boldsymbol{\omega}=\mathbf{0}\ \text{and}\ \omega_j>0\ \text{for all }j \Bigr\},
\]
where $\mathbf{A}$ is the coefficient matrix in \eqref{eq:vandermonde}. The goal of this section is to
show that $\mathcal{W}_{>0}(Z_n,Z_m)\neq\emptyset$ precisely under strict interlacing, and to describe
all such $\boldsymbol{\omega}$ explicitly. The first step is to make the Vandermonde structure of the
system completely transparent.

\begin{lemma}\label{lem:vandermonde-reduction}
Assume the hypotheses and notation of Lemma~\ref{lemma:OPRLlemma}. Consider the $m\times n$ matrix $\mathbf{A}=(A_{k,j})_{1\le k\le m,\ 1\le j\le n}$ with entries
\[
A_{k,j}=x_j^{k-1}\,P_m(x_j),\quad 1\le k\le m,\ 1\le j\le n.
\]
Then $\mathbf{A}$ factorizes as
\[
\mathbf{A}=\mathbf{V}\mathbf{D},
\]
where $\mathbf{V}$ is the $m\times n$ Vandermonde matrix
\[
\mathbf{V}=
\begin{bmatrix}
\;1 & 1 & \cdots & 1\;\\[3pt]
\;x_1 & x_2 & \cdots & x_n\;\\[3pt]
\:\vdots & \vdots & & \vdots\;\\[3pt]
\;x_1^{m-1} & x_2^{m-1} & \cdots & x_n^{m-1}\;
\end{bmatrix},
\]
and $\mathbf{D}=\mathrm{diag}(P_m(x_1),\dots,P_m(x_n))$. In particular, $P_m(x_j)\neq 0$ for all $j$,
so $\mathbf{D}$ is invertible, and therefore
\[
\operatorname{rank}\mathbf{A}=m,\quad \dim\ker\mathbf{A}=n-m.
\]
Moreover, \eqref{eq:vandermonde} is exactly the matrix equation $\mathbf{A}\boldsymbol{\omega}=\mathbf{0}$.
\end{lemma}

\begin{proof}
The factorization $\mathbf{A}=\mathbf{V}\mathbf{D}$ is immediate from the definitions:
$(\mathbf{V}\mathbf{D})_{k,j}=x_j^{k-1}P_m(x_j)=A_{k,j}$.

Since $Z_n\cap Z_m=\emptyset$, we have $P_m(x_j)\neq 0$ for every $j$, hence $\mathbf{D}$ is invertible.
Because the nodes $x_1,\dots,x_n$ are pairwise distinct, the Vandermonde determinant formula implies
that any $m\times m$ square Vandermonde submatrix obtained by selecting $m$ columns of $\mathbf{V}$
has nonzero determinant, and thus $\mathbf{V}$ has full row rank $m$. Consequently,
\[
\operatorname{rank}\mathbf{A}=\operatorname{rank}(\mathbf{V}\mathbf{D})=\operatorname{rank}\mathbf{V}=m,
\]
and since $\mathbf{A}$ is $m\times n$, it follows that $\dim\ker\mathbf{A}=n-m$.
Finally, the relations in \eqref{eq:vandermonde} are precisely the coordinates of
$\mathbf{A}\boldsymbol{\omega}=\mathbf{0}$.
\end{proof}

We next describe the full solution space of \eqref{eq:vandermonde} with no sign restrictions.
This is elementary and included only to fix notation for the positivity analysis later on.
Since \(\mathbf A=\mathbf V\mathbf D\) (Lemma~\ref{lem:vandermonde-reduction}) with \(\mathbf D\) invertible,
the problem reduces to the nullspace of the Vandermonde matrix \(\mathbf V\).
The explicit sparse generators below are the standard barycentric weights from Lagrange interpolation,
see \cite[Ch.~2, \S2.1]{StoerBulirsch}. The spanning statement is the usual fact that every linear
dependence is a linear combination of minimal ones (circuits); see \cite[Ch.~1, \S1.1]{OxleyMatroid}.
(Our indexing differs from \cite{StoerBulirsch}, but the formulas translate directly.)

\begin{lemma}\label{lem:circuit-solutions}
Assume the hypotheses and notation of Lemma~\ref{lem:vandermonde-reduction}.
For any subset \(\mathcal{J}\subset\{1,\dots,n\}\) with \(|\mathcal{J}|=m+1\), write
\(\mathcal{J}=\{j_0,\dots,j_m\}\) and set
\[
Q_{\mathcal{J}}(x)=(x-x_{j_0})\cdots(x-x_{j_m}).
\]
Define \(\boldsymbol{\omega}^{(\mathcal{J})}\in\mathbb R^n\) by
\begin{equation}
\label{eq:circuit-solution}
\omega^{(\mathcal{J})}_j=
\begin{cases}
\dfrac{1}{P_m(x_j)\,Q_{\mathcal{J}}'(x_j)}, & j\in \mathcal{J},\\[7pt]
0, & j\notin \mathcal{J}.
\end{cases} 
\end{equation}
Then \(\boldsymbol{\omega}^{(\mathcal{J})}\neq\mathbf 0\) and \(\mathbf A\,\boldsymbol{\omega}^{(\mathcal{J})}=\mathbf 0\).
Moreover,
\[
\ker \mathbf A=\mathrm{span}\bigl\{\boldsymbol{\omega}^{(\mathcal{J})}:\ \mathcal{J}\subset\{1,\dots,n\},\ |\mathcal{J}|=m+1\bigr\}.
\]
\end{lemma}

\begin{proof}
Fix \(\mathcal{J}\) and define \[\boldsymbol{\lambda}^{(\mathcal{J})}=\mathbf D\,\boldsymbol{\omega}^{(\mathcal{J})}.\]
By construction, \(\boldsymbol{\lambda}^{(\mathcal{J})}\) is supported on \(\mathcal{J}\) and, in addition, its nonzero entries are
\(\lambda^{(\mathcal{J})}_j=1/Q_{\mathcal{J}}'(x_j)\) for \(j\in\mathcal{J}\).
These are exactly the barycentric weights associated with the nodes \(\{x_j\}_{j\in\mathcal{J}}\);
in particular, they satisfy the Vandermonde moment identities
\[
\sum_{j\in\mathcal{J}}\frac{x_j^{k}}{Q_{\mathcal{J}}'(x_j)}=0,\quad k=0,1,\dots,m-1,
\]
see \cite[Ch.~2, \S2.1]{StoerBulirsch}. This is precisely \(\mathbf V\,\boldsymbol{\lambda}^{(\mathcal{J})}=\mathbf 0\),
and therefore
\[
\mathbf A\,\boldsymbol{\omega}^{(\mathcal{J})}
=\mathbf V\mathbf D\,\boldsymbol{\omega}^{(\mathcal{J})}
=\mathbf V\,\boldsymbol{\lambda}^{(\mathcal{J})}
=\mathbf 0.
\]
Since the entries on \(\mathcal{J}\) are nonzero, also \(\boldsymbol{\omega}^{(\mathcal{J})}\neq\mathbf 0\).

For the spanning claim, \(\mathbf D\) is invertible, so \(\ker\mathbf A=\mathbf D^{-1}\ker\mathbf V\).
Thus it suffices to note that \(\ker\mathbf V\) is spanned by its minimal-support dependencies.
In matroid language, the column dependencies of \(\mathbf V\) define a matroid, and its circuits are
the minimal dependent sets; see \cite[Ch.~1, \S1.1]{OxleyMatroid}. Here every \(\mathcal{J}\) with
\(|\mathcal{J}|=m+1\) supports such a circuit vector \(\boldsymbol{\lambda}^{(\mathcal{J})}\), and the standard
circuit-decomposition argument shows that any \(\boldsymbol{\lambda}\in\ker\mathbf V\) is a linear combination of
these circuit vectors. Applying \(\mathbf D^{-1}\) yields the stated spanning set for \(\ker\mathbf A\).
\end{proof}

We now use strict interlacing to describe the nonnegative and strictly positive solutions of
\eqref{eq:vandermonde}. By Lemma~\ref{lem:circuit-solutions}, \(\ker\mathbf A\) is spanned by sparse
circuit vectors indexed by subsets \(\mathcal J\subset\{1,\dots,n\}\) with \(|\mathcal J|=m+1\).
Thus, the positivity analysis reduces to identifying which circuits can be scaled to be entrywise
nonnegative. Under strict interlacing, this becomes transparent after grouping the indices
\(\{1,\dots,n\}\) into the interlacing bands determined by the relative position of \(Z_m\) among
\(Z_n\). Indeed, assume that $Z_m$ strictly interlaces $Z_n$ and let
\[
i_0=0<i_1<\cdots<i_m<n=i_{m+1}
\]
be the interlacing indices, so that $x_{i_k}<y_k<x_{i_k+1}$ for $k=1,\dots,m$. These indices induce
a partition of $\{1,\dots,n\}$ into $m+1$ consecutive \emph{bands} of indices,
\[
I_r=\{i_r+1,\dots,i_{r+1}\},\quad r=0,\dots,m,
\]
so that $\{1,\dots,n\}=\bigsqcup_{r=0}^m I_r$. Since $P_m$ is monic with simple real zeros, its sign is constant on each open interval between consecutive zeros, and hence
$P_m(x_j)$ has a constant sign for all $j$ within a fixed band $I_r$ since the zeros $y_k$ lie between
bands. In particular, the sign alternates from band to band as one crosses a zero of $P_m$.

\begin{eje}\label{ex:bands-signs}
Let $n=7$ and $m=3$, and consider
\[
Z_n=\{x_1,\dots,x_7\}=\{0,1,2,3,4,5,6\},\quad
Z_m=\{y_1,y_2,y_3\}=\Bigl\{\frac12,\frac52,\frac92\Bigr\}.
\]
Then $Z_m$ strictly interlaces $Z_n$ with interlacing indices
\[
i_0=0<i_1=1<i_2=3<i_3=5<i_4=7,
\]
since $y_1 \in (x_1,x_2)$, $y_2 \in (x_3,x_4)$ and $y_3 \in (x_5,x_6)$. Then, the associated bands are
\[
I_0=\{1\},\quad I_1=\{2,3\},\quad I_2=\{4,5\},\quad I_3=\{6,7\},
\]
and they partition $\{1,\dots,7\}$. Moreover,
\[
P_3(x)=(x-y_1)(x-y_2)(x-y_3)=\Bigl(x-\frac12\Bigr)\Bigl(x-\frac52\Bigr)\Bigl(x-\frac92\Bigr).
\]
Thus, $P_3(x_j)$ has constant sign on each band: it is negative on $I_0$, positive on $I_1$,
negative on $I_2$, and positive on $I_3$.
\end{eje} 
We leverage this interlacing structure to characterize the strictly positive solutions of the Vandermonde system.

\begin{lemma}\label{lem:positive-cone}
Assume the hypotheses and notation of Lemma~\ref{lem:circuit-solutions}. Assume, in addition, that
$Z_m$ strictly interlaces $Z_n$ in the sense of Definition~\ref{def:strict-interlacing-oprl}, and let
\[
i_0=0<i_1<\cdots<i_m<n=i_{m+1}
\]
be the associated indices of interlacing. Set $I_r=\{i_r+1,\dots,i_{r+1}\}$ for $r=0,\dots,m$ so that $\{1,\dots,n\}=\bigsqcup_{r=0}^m I_r$\footnote{We write $\sqcup$ for disjoint union.}. Define the family
\[
\mathfrak{J}_{+}
=
\Bigl\{\mathcal{J}\subset\{1,\dots,n\}:\ |\mathcal{J}|=m+1,\ |\mathcal{J}\cap I_r|=1
\ \text{for all }r=0,\dots,m\Bigr\}.
\]
For each $\mathcal{J}\in\mathfrak{J}_{+}$, let $\boldsymbol{\omega}_{\ge0}=\boldsymbol{\omega}^{(\mathcal{J})}$ be the
vector given by \eqref{eq:circuit-solution}. Then, $\boldsymbol{\omega}_{\ge0}\ge 0$. Moreover, the following hold:
\begin{enumerate}
\item[(a)] Define
$\mathcal{W}_{\ge 0}(Z_n,Z_m)=\{\boldsymbol{\omega}\in\mathbb{R}^n:\mathbf{A}\boldsymbol{\omega}=\mathbf{0}
\ \text{and}\ \omega_j\ge 0\ \text{for all }j\}$.
Then
\[
\mathcal{W}_{\ge 0}(Z_n,Z_m)
=
\operatorname{cone}\bigl\{\boldsymbol{\omega}^{(\mathcal{J})}:\ \mathcal{J}\in\mathfrak{J}_{+}\bigr\}.
\]

\item[(b)] One has
\[
\mathcal{W}_{>0}(Z_n,Z_m)=\mathcal{W}_{\ge0}(Z_n,Z_m)\cap \mathbb{R}^n_{>0}
=\operatorname{relint}\bigl(\mathcal{W}_{\ge 0}(Z_n,Z_m)\bigr),
\]
where $\operatorname{relint}A$ denotes the relative interior of $A$, taken (for our purposes) in $\ker\mathbf{A}$.
Furthermore, $\boldsymbol{\omega}\in\mathcal{W}_{>0}(Z_n,Z_m)$ if and only if
\[
\boldsymbol{\omega}=\sum_{\mathcal{J}\in\mathfrak{J}_{+}} t_{\mathcal{J}}\,
\boldsymbol{\omega}^{(\mathcal{J})},
\]
and, for every $j\in\{1,\dots,n\}$, there exists at least one $\mathcal{J}\in\mathfrak{J}_{+}$ such that
$j\in\mathcal{J}$ and $t_{\mathcal{J}}>0$. In particular, $\mathcal{W}_{>0}(Z_n,Z_m)\neq\emptyset$.
\end{enumerate}
\end{lemma}
\begin{proof}
Fix $\mathcal{J}\in\mathfrak{J}_{+}$. Write $\mathcal{J}=\{j_0,\dots,j_m\}$ with $j_0<\cdots<j_m$
and set
\[
Q_{\mathcal{J}}(x)=(x-x_{j_0})\cdots(x-x_{j_m}).
\]
By \eqref{eq:circuit-solution}, for $r=0,\dots,m$ one has
\[
\omega^{(\mathcal{J})}_{j_r}=\frac{1}{P_m(x_{j_r})\,Q_{\mathcal{J}}'(x_{j_r})}.
\]
The rest of the entries are zero. We show that the nonzero entries of $\boldsymbol{\omega}^{(\mathcal{J})}$ all have the same
sign. Indeed, since $j_0<\cdots<j_m$, we have
\[
Q_{\mathcal{J}}'(x_{j_r})
=\left. \dfrac{Q_{\mathcal{J}}(x)}{(x-x_{j_r})} \right|_{x=x_{j_r}}
\]
and exactly $m-r$ factors are negative, while the remaining $r$ factors are positive. Hence
\[
\operatorname{sgn}\bigl(Q_{\mathcal{J}}'(x_{j_r})\bigr)=(-1)^{m-r}.
\]
On the other hand, strict interlacing gives $y_r\in(x_{i_r},x_{i_r+1})$ for $r=1,\dots,m$. Then, by the
definition of the bands $I_r=\{i_r+1,\dots,i_{r+1}\}$, it follows that, for each $j\in I_r$,
\[
y_r<x_j<y_{r+1},
\quad r=0,\dots,m,
\]
with the conventions $y_0=-\infty$ and $y_{m+1}=+\infty$. Since $P_m$ is monic with simple real zeros
$y_1<\cdots<y_m$, its sign alternates between successive zeros, and therefore
\[
\operatorname{sgn}\bigl(P_m(x_j)\bigr)=(-1)^{m-r},\quad j\in I_r.
\]
Now, because $\mathcal{J}\in\mathfrak{J}_{+}$ meets each $I_r$ in exactly one index, we have
$j_r\in I_r$ for $r=0,\dots,m$, and hence
\[
\operatorname{sgn}\bigl(P_m(x_{j_r})\,Q_{\mathcal{J}}'(x_{j_r})\bigr)
=(-1)^{m-r}\cdot(-1)^{m-r}=+1.
\]
Thus $\omega^{(\mathcal{J})}_{j_r}>0$ for $r=0,\dots,m$. This proves that $\boldsymbol{\omega}^{(\mathcal{J})}\ge0$. Now we prove statements (a) and (b).

\smallskip
(a).
For every $\mathcal{J}\in\mathfrak{J}_{+}$, Lemma~\ref{lem:circuit-solutions} gives
$\mathbf{A}\boldsymbol{\omega}^{(\mathcal{J})}=\mathbf{0}$, and the previous paragraph shows that
$\boldsymbol{\omega}^{(\mathcal{J})}\ge0$ after the global sign choice. Hence
\[
\operatorname{cone}\bigl\{\boldsymbol{\omega}^{(\mathcal{J})}:\ \mathcal{J}\in\mathfrak{J}_{+}\bigr\}
\subset \mathcal{W}_{\ge0}(Z_n,Z_m).
\]

For the reverse inclusion, note that $\mathcal{W}_{\ge0}(Z_n,Z_m)=\{\boldsymbol{\omega}:\mathbf{A}\boldsymbol{\omega}=\mathbf{0},
\ \boldsymbol{\omega}\ge0\}$ is a polyhedral cone. By standard polyhedral theory, it is generated by
its extreme rays, and each extreme ray admits a representative with inclusion-minimal support among
nonzero vectors in the cone; see, e.g., \cite[Chs.~7--8]{SchrijverLP} (extreme rays and minimal
support feasible directions). In our setting, the minimal-support elements of $\ker\mathbf{A}$ are
exactly the circuit vectors from Lemma~\ref{lem:circuit-solutions}, i.e., vectors supported on
$\mathcal{J}\subset\{1,\dots,n\}$ with $|\mathcal{J}|=m+1$. Among these circuit vectors, the previous
sign computation shows that a circuit can be chosen nonnegative if and only if its support meets each
band $I_r$ exactly once, namely $\mathcal{J}\in\mathfrak{J}_{+}$: if two indices fall in the same
band (or some band is missed), then the alternation of $\operatorname{sgn}(P_m(x_j))$ across the bands
cannot match the fixed alternation of $\operatorname{sgn}(Q_{\mathcal{J}}'(x_j))$ along the ordered
support, and the resulting circuit has mixed signs, hence no global sign makes it nonnegative.
Therefore, the extreme rays of $\mathcal{W}_{\ge0}(Z_n,Z_m)$ are precisely generated by the vectors
$\boldsymbol{\omega}^{(\mathcal{J})}$ with $\mathcal{J}\in\mathfrak{J}_{+}$, and (a) follows.

\smallskip
(b).
The equality
\[
\mathcal{W}_{>0}(Z_n,Z_m)=\mathcal{W}_{\ge0}(Z_n,Z_m)\cap \mathbb{R}^n_{>0}
\]
is immediate from the definitions. Moreover, since $\mathcal{W}_{\ge0}(Z_n,Z_m)$ is the intersection
of the subspace $\ker\mathbf{A}$ with the coordinate half-spaces $\{\omega_j\ge0\}$, its relative
interior in $\ker\mathbf{A}$ consists of those points for which all nonredundant inequalities are
strict; see, e.g., \cite[Chs.~7--8]{SchrijverLP}. Under strict interlacing, no coordinate
inequality is redundant because, as noted above, for every $j\in\{1,\dots,n\}$ one can choose
$\mathcal{J}\in\mathfrak{J}_{+}$ with $j\in\mathcal{J}$, so there exists some
$$\boldsymbol{\omega}^{(\mathcal{J})}\in\mathcal{W}_{\ge0}(Z_n,Z_m)$$ with $\omega^{(\mathcal{J})}_j>0$.
Hence,
\[
\operatorname{relint}\bigl(\mathcal{W}_{\ge0}(Z_n,Z_m)\bigr)
=\mathcal{W}_{\ge0}(Z_n,Z_m)\cap \mathbb{R}^n_{>0}
=\mathcal{W}_{>0}(Z_n,Z_m).
\]
Next, by (a), any $\boldsymbol{\omega}\in\mathcal{W}_{\ge0}(Z_n,Z_m)$ admits a representation
\[
\boldsymbol{\omega}=\sum_{\mathcal{J}\in\mathfrak{J}_{+}} t_{\mathcal{J}}\,
\boldsymbol{\omega}^{(\mathcal{J})},
\quad t_{\mathcal{J}}\ge0.
\]
If $\boldsymbol{\omega}\in\mathcal{W}_{>0}(Z_n,Z_m) \subset {W}_{\ge0}(Z_n,Z_m)$ and there existed $j$ such that
$t_{\mathcal{J}}=0$ for every $\mathcal{J}\in\mathfrak{J}_{+}$ with $j\in\mathcal{J}$, then the
$j$-th component would satisfy
\[
\omega_j=\sum_{\mathcal{J}\ni j} t_{\mathcal{J}}\,\omega^{(\mathcal{J})}_j=0,
\]
contradicting $\omega_j>0$. Conversely, if the above representation holds and for every
$j\in\{1,\dots,n\}$ there exists $\mathcal{J}\in\mathfrak{J}_{+}$ with $j\in\mathcal{J}$ and
$t_{\mathcal{J}}>0$, then
\[
\omega_j=\sum_{\mathcal{J}\ni j} t_{\mathcal{J}}\,\omega^{(\mathcal{J})}_j>0,
\]
since each $\omega^{(\mathcal{J})}_j\ge0$ and at least one summand is strictly positive. This gives
the stated characterization of $\mathcal{W}_{>0}(Z_n,Z_m)$.

Finally, to see $\mathcal{W}_{>0}(Z_n,Z_m)\neq\emptyset$, for each $j\in\{1,\dots,n\}$ choose some
$\mathcal{J}_j\in\mathfrak{J}_{+}$ with $j\in\mathcal{J}_j$. This is always possible: choose an arbitrary index from each band other than the one containing $j$, and in that band select $j$ itself. Then,
\[
\boldsymbol{\omega}=\sum_{j=1}^{n}\boldsymbol{\omega}^{(\mathcal{J}_j)}
\]
belongs to $\mathcal{W}_{\ge0}(Z_n,Z_m)$ and has $\omega_j>0$ for every $j$, hence
$\boldsymbol{\omega}\in\mathcal{W}_{>0}(Z_n,Z_m)$.
\end{proof}

We can now combine the quadrature reformulation (Lemma~\ref{lemma:OPRLlemma}),
the description of \(\ker\mathbf{A}\) (Lemma~\ref{lem:circuit-solutions}), and the positivity
analysis under strict interlacing (Lemma~\ref{lem:positive-cone}) to solve
Problem~\ref{prob:underdetermined-oprl}.

\begin{theorem}\label{thm:OPRL-solution}
Assume the notation and hypotheses of Problem~\ref{prob:underdetermined-oprl}.
Then, it admits a solution if and only if
\(Z_m\) strictly interlaces \(Z_n\) in the sense of Definition~\ref{def:strict-interlacing-oprl}.
In this case, the solution is not unique whenever \(m<n-1\).  More precisely, fix any \(\boldsymbol{\omega}\in\mathcal{W}_{>0}(Z_n,Z_m)\) as in
Lemma~\ref{lem:positive-cone}, and consider the discrete measure
\(\mu=\sum_{j=1}^n \omega_j\,\delta_{x_j}\).
Then the quadrature moments
\[
\mu_k=\sum_{j=1}^n \omega_j\,x_j^k,\quad k=0,\dots,2n-1,
\]
determine a Jacobi matrix \(J_n\) (equivalently, the first \(n\) three-term recurrence coefficients)
with
\[
\sigma(J_n)=Z_n, \quad
\sigma(J_m)=Z_m,
\]
where \(J_m\) denotes the leading principal submatrix of \(J_n\) of order \(m\).
Conversely, every Jacobi matrix \(J_n\) satisfying \(\sigma(J_n)=Z_n\) and \(\sigma(J_m)=Z_m\)
arises from a choice of weights \(\boldsymbol{\omega}\in\mathcal{W}_{>0}(Z_n,Z_m)\) through the
construction in Lemma~\ref{lemma:OPRLlemma} and the positivity mechanism in
Lemma~\ref{lem:positive-cone}.
\end{theorem}
\begin{proof}
(Necessity.)
Assume there exists a Jacobi matrix \(J_n\) such that \(\sigma(J_n)=Z_n\) and
\(\sigma(J_m)=Z_m\), where \(J_m\) is the leading principal submatrix of order \(m\).
By the standard eigenvalue interlacing theorem for Hermitian principal submatrices,
the spectrum of \(J_m\) strictly interlaces the spectrum of \(J_n\). Hence \(Z_m\)
strictly interlaces \(Z_n\) in the sense of Definition~\ref{def:strict-interlacing-oprl}.

(Sufficiency.)
Assume now that \(Z_m\) strictly interlaces \(Z_n\). By Lemma~\ref{lem:positive-cone},
the set \(\mathcal{W}_{>0}(Z_n,Z_m)\) is nonempty, so we may choose
\(\boldsymbol{\omega}=(\omega_1,\dots,\omega_n)\in\mathcal{W}_{>0}(Z_n,Z_m)\).
Equivalently, \(\boldsymbol{\omega}\) is a strictly positive solution of the Vandermonde system
\eqref{eq:vandermonde}. Lemma~\ref{lemma:OPRLlemma} then yields a positive Borel measure
\(\mu=\sum_{j=1}^n\omega_j\,\delta_{x_j}\) for which the monic OPRL satisfy
\(P_n(\cdot;\mu)=\prod_{j=1}^n(x-x_j)\) and \(P_m(\cdot;\mu)=\prod_{k=1}^m(x-y_k)\).
By the same lemma, this is equivalent to the existence of a Jacobi matrix \(J_n\) with
\(\sigma(J_n)=Z_n\) and \(\sigma(J_m)=Z_m\). This proves solvability.

Finally, the description of solutions follows from the construction in the proof of Lemma~\ref{lemma:OPRLlemma}: given \(\boldsymbol{\omega}\in\mathcal{W}_{>0}(Z_n,Z_m)\), the moments
\(\mu_k=\sum_{j=1}^n\omega_j x_j^k\) for \(k=0,\dots,2n-1\) coincide with the moments of \(\mu\) up to
order \(2n-1\), and the first \(n\) recurrence coefficients (hence \(J_n\)) are recovered from these
moments by the classical Stieltjes procedure. Conversely, any solution \(J_n\) determines an OPRL
family and hence a discrete quadrature representation at the zeros of \(P_n\), producing weights
\(\boldsymbol{\omega}\in\mathcal{W}_{>0}(Z_n,Z_m)\) and the corresponding moment data. 
\end{proof}

As a consequence of Theorem~\ref{thm:OPRL-solution}, we obtain the following extension of
Wendroff's theorem beyond consecutive degrees.

\begin{coro}[Generalized Wendroff  on $\mathbb{R}$]\label{cor:gen-Wendroff}
Fix integers \(1\le m<n\) and let
\[
Z_n=\{x_k\}_{k=1}^{n}, \quad Z_m=\{y_k\}_{k=1}^{m}.
\]
be sets of distinct real numbers with \(x_1<\cdots<x_n\) and \(y_1<\cdots<y_m\).
Then, there exists a positive Borel measure \(\mu\) on \(\mathbb R\) and a monic OPRL family
\(\{P_k(\cdot;\mu)\}_{k\ge 0}\) such that
\[
\mathrm{zeros}\bigl(P_n(\cdot;\mu)\bigr)=Z_n,\quad 
\mathrm{zeros}\bigl(P_m(\cdot;\mu)\bigr)=Z_m
\]
if and only if \(Z_m\) strictly interlaces \(Z_n\) in the sense of
Definition~\ref{def:strict-interlacing-oprl}.
\end{coro}

\begin{proof}
{  This follows from Theorem~\ref{thm:OPRL-solution} and Lemma~\ref{lemma:OPRLlemma}, together with Favard's theorem.}
\end{proof}

Of course, describing \emph{all} strictly positive solutions \(\boldsymbol{\omega}\in\mathcal{W}_{>0}(Z_n,Z_m)\)
can hide a nontrivial combinatorial step: one must understand which nonnegative circuit vectors
\(\boldsymbol{\omega}^{(\mathcal{J})}\) may be combined so as to obtain a vector with all entries strictly
positive. This depends strongly on the fine structure of the interlacing (in particular, on the bands
determined by the indices \(i_0,i_1\dots,i_{m},i_{m+1}\)). While there should be a general procedure ensuring that
all admissible combinations are taken into account whenever strict interlacing holds, developing such a
global enumeration scheme is not the purpose of the present paper. On the other hand, for specific
interlacing patterns it is straightforward to devise an explicit strategy to generate all admissible
weights.

Nonetheless, in Section~\ref{sec:algos} we present an algorithm, together with
Example~\ref{ex:oprl}, illustrating how the above theory can be implemented to describe the family of
Jacobi matrices \(J_n\) realizing the prescribed spectral data. In the next section, we establish the
unit-circle analogue in the POPUC setting.

\section{Main results on $\mathbb S^1$}\label{sec:POPUC}
We now treat the POPUC case, following the same strategy as in the OPRL setting. Fix integers $1\le m<n$ and two sets
$Z_n,Z_m\subset\mathbb{S}^1$ as in Lemma~\ref{lemma:POPUClemma}, with $Z_n\cap Z_m=\emptyset$.
Write
\[
Z_n=\{\zeta_j=e^{i\theta_j}:1\le j\le n\},\quad
Z_m=\{\xi_k=e^{i\varphi_k}:1\le k\le m\}.
\]

As explained in Section~\ref{sec:prelim}, Lemma \ref{lem:vandermonde-reduction-POPUC}, solvability is equivalent to the existence of
weights $\omega_1,\dots,\omega_n>0$ satisfying the POPUC Vandermonde-type system
\eqref{eq:vandermondePOPUC}. Accordingly, we introduce the set of strictly positive solutions
\[
\mathcal{W}_{>0}(Z_n,Z_m)
=
\Bigl\{ \boldsymbol{\omega}=(\omega_1,\dots,\omega_n)\in\mathbb{R}^n:
\mathbf{A}\boldsymbol{\omega}=\mathbf{0}\ \text{and}\ \omega_j>0\ \text{for all }j \Bigr\},
\]
where $\mathbf{A}$ is the coefficient matrix in \eqref{eq:vandermondePOPUC} written in the form
$\mathbf{A}\boldsymbol{\omega}=\mathbf{0}$ (Lemma~\ref{lem:vandermonde-reduction-POPUC} below).
The goal of this section is to show that $\mathcal{W}_{>0}(Z_n,Z_m)\neq\emptyset$ precisely under
strict interlacing, and to describe all such $\boldsymbol{\omega}$ explicitly. As in the OPRL case,
the first step is to make the Vandermonde structure transparent.

\begin{lemma}\label{lem:vandermonde-reduction-POPUC}
Assume the hypotheses and notation of Lemma~\ref{lemma:POPUClemma}. Consider the $(m-1)\times n$
matrix $\mathbf{A}=(A_{k,j})_{1\le k\le m-1,\;1\le j\le n}$ with entries
\[
A_{k,j}
=\zeta_j^{\,k-1}\,\frac{P_m(\zeta_j)}{\zeta_j^{\,m-1}},
\quad 1\le k\le m-1,\;1\le j\le n.
\]
Then $\mathbf{A}$ factorizes as
\[
\mathbf{A}=\mathbf{V}\mathbf{D},
\]
where $\mathbf{V}$ is the $(m-1)\times n$ complex Vandermonde matrix
\[
\mathbf{V}=
\begin{bmatrix}
\;1 & 1 & \cdots & 1\;\\[3pt]
\;\zeta_1 & \zeta_2 & \cdots & \zeta_n\;\\[3pt]
\;\vdots & \vdots & & \vdots\;\\[3pt]
\;\zeta_1^{\,m-2} & \zeta_2^{\,m-2} & \cdots & \zeta_n^{\,m-2}\;
\end{bmatrix},
\]
and $\mathbf{D}=\operatorname{diag}\bigl(\zeta_1^{1-m}P_m(\zeta_1),\dots,\zeta_n^{1-m}P_m(\zeta_n)\bigr)$.
In particular, $P_m(\zeta_j)\neq0$ for all $j$, so $\mathbf{D}$ is invertible, and therefore
\[
\operatorname{rank}\mathbf{A}=m-1,\quad \dim\ker\mathbf{A}=n-m+1.
\]
Moreover, \eqref{eq:vandermondePOPUC} is exactly the matrix equation $\mathbf{A}\boldsymbol{\omega}=\mathbf{0}$.
\end{lemma}

\begin{proof}
Starting from \eqref{eq:vandermondePOPUC},
\[
\sum_{j=1}^n \omega_j\,P_m(\zeta_j)\,\zeta_j^{-k}=0,\quad k=1,\dots,m-1,
\]
rewrite $\zeta_j^{-k}=\zeta^{m-1-k}_j\zeta^{1-m}_j$ to obtain
\[
\sum_{j=1}^n \omega_j\,\frac{P_m(\zeta_j)}{\zeta_j^{m-1}}\,\zeta_j^{m-1-k}=0,\quad k=1,\dots,m-1.
\]
Relabeling $k \mapsto m-1-k$ gives precisely $\mathbf{A}\boldsymbol{\omega}=\mathbf{0}$ with the stated entries.

The factorization $\mathbf{A}=\mathbf{V}\mathbf{D}$ is immediate from the definitions.
Since $Z_n\cap Z_m=\emptyset$, we have $P_m(\zeta_j)\neq0$ for every $j$, hence $\mathbf{D}$ is invertible.
Because the nodes $\zeta_1,\dots,\zeta_n$ are pairwise distinct, the usual Vandermonde determinant argument
implies that $\mathbf{V}$ has full row rank $m-1$.
Consequently,
\[
\operatorname{rank}\mathbf{A}=\operatorname{rank}(\mathbf{V}\mathbf{D})=\operatorname{rank}\mathbf{V}=m-1,
\]
and therefore $\dim\ker\mathbf{A}=n-(m-1)=n-m+1$.
\end{proof}

We next describe the full solution space of \eqref{eq:vandermondePOPUC} with no sign restrictions.
This is the unit-circle analogue of Lemma~\ref{lem:circuit-solutions}. Here the matrix $\mathbf{A}$
is complex-valued, but we will obtain a spanning family of \emph{real} circuit vectors; this is the
form needed for the positivity analysis.

\begin{lemma}\label{lem:circuit-solutions-POPUC}
Assume the hypotheses and notation of Lemma~\ref{lem:vandermonde-reduction-POPUC}.
For any subset \(\mathcal{J}\subset\{1,\dots,n\}\) with \(|\mathcal{J}|=m\), write \(\mathcal{J}=\{j_1,\dots,j_{m}\}\) and set
\[
Q_{\mathcal{J}}(z)=(z-\zeta_{j_1})\cdots(z-\zeta_{j_{m}}).
\]
Define \(\boldsymbol{\omega}^{(\mathcal{J})}\in\mathbb R^n\) by
\begin{equation}\label{eq:circuit-solution-POPUC-sine}
\omega^{(\mathcal{J})}_j=
\begin{cases}
\displaystyle
\frac{1}{
\displaystyle
\prod_{k=1}^m \sin\!\left(\dfrac{\theta_j-\varphi_k}{2}\right)\,
\prod_{\substack{r=0\\ j_r\neq j}}^{m}
\sin\!\left(\dfrac{\theta_j-\theta_{j_r}}{2}\right)}, & j\in \mathcal{J},\\[7pt]
0, & j\notin \mathcal{J}.
\end{cases}
\end{equation}
Then \(\boldsymbol{\omega}^{(\mathcal{J})}\neq\mathbf 0\) and \(\mathbf A\,\boldsymbol{\omega}^{(\mathcal{J})}=\mathbf 0\).
Moreover,
\[
\ker \mathbf A=\mathrm{span}\bigl\{\boldsymbol{\omega}^{(\mathcal{J})}:\ \mathcal{J}\subset\{1,\dots,n\},\ |\mathcal{J}|=m\bigr\}.
\]
In particular, \(\ker\mathbf A\cap\mathbb R^n\) is spanned over \(\mathbb R\) by the same family.
\end{lemma}

\begin{proof}
Since \(\mathbf A=\mathbf V\mathbf D\) (Lemma~\ref{lem:vandermonde-reduction-POPUC}) with \(\mathbf D\) invertible,
it suffices to describe \(\ker\mathbf V\). The standard Vandermonde argument gives, for each
\(|\mathcal{J}|=m\), a nonzero vector \(\boldsymbol{\lambda}^{(\mathcal{J})}\in\ker\mathbf V\) supported on \(\mathcal{J}\) with
\[
\lambda^{(\mathcal{J})}_j=
\begin{cases}
\dfrac{1}{Q_\mathcal{J}'(\zeta_j)}, & j\in \mathcal{J},\\[7pt]
0, & j\notin \mathcal{J},
\end{cases}
\]
and these circuit vectors span \(\ker\mathbf V\); see \cite[Ch.~2, \S2.1]{StoerBulirsch} for the barycentric
weights and \cite[Ch.~1, \S1.1]{OxleyMatroid} for the circuit-decomposition principle. Applying \(\mathbf D^{-1}\),
we obtain a spanning family of \(\ker\mathbf A\) given (up to a nonzero scalar factor) by
\[
\widetilde\omega^{(\mathcal{J})}_j=
\begin{cases}
\displaystyle
\frac{\zeta_j^{\,m-1}}{P_m(\zeta_j)\,Q_\mathcal{J}'(\zeta_j)}, & j\in \mathcal{J},\\[7pt]
0, & j\notin \mathcal{J}.
\end{cases}
\]
We now rewrite \(\widetilde\omega^{(\mathcal{J})}\) in terms of half-angle sines. For arbitrary $\zeta=\exp{i\theta}$ and $\xi = \exp{i\varphi}$, consider the identity
\begin{equation*}
    \begin{split}
        \zeta-\xi &= 2i\,\exp\left({i\dfrac{\theta+\varphi}{2}}\right)\sin\!\left(\frac{\theta-\varphi}{2}\right).
    \end{split}
\end{equation*}
From that, obtain
\[
P_m(\zeta_j)
=(2i)^m \exp\left({\dfrac{i}{2}\left(m\theta_j+\sum_{k=1}^m\varphi_k\right)}\right)
\prod_{k=1}^m \sin\!\left(\frac{\theta_j-\varphi_k}{2}\right),
\]
together with
\[
Q_\mathcal{J}'(\zeta_j)
=(2i)^{m-1} \exp\left({\dfrac{i}{2}\left((m-2)\theta_j+\sum_{r = 1}^m\theta_{j_r}\right)}\right)
\prod_{r \neq s}^m \sin\!\left(\frac{\theta_j-\theta_{j_r}}{2}\right),
\]
whenever $j=j_s \in \mathcal{J}$. Then, it is direct to check that $\omega_j^{(\mathcal{J})}=C_{\mathcal{J}}\widetilde \omega^{(\mathcal{J})}_j$ where
\[
C_{\mathcal{J}}=(2i)^{1-2m}\exp\left(-\dfrac{i}{2}\left(\sum^{m}_{k=0}\varphi_k+\theta_{j_k}\right)\right),
\]
which is unimodular and depends only on \(\mathcal{J}\) and not on the specific index $j$. Therefore \(\widetilde{\boldsymbol{\omega}}^{(\mathcal{J})}\) differs from
\(\boldsymbol{\omega}^{(\mathcal{J})}\) in \eqref{eq:circuit-solution-POPUC-sine} by multiplication with a nonzero complex scalar,
and hence \(\boldsymbol{\omega}^{(\mathcal{J})}\in\ker\mathbf A\). The spanning statement follows from the spanning property of
the circuit family before rescaling. Since the spanning family in \eqref{eq:circuit-solution-POPUC-sine} is real, it
also spans \(\ker\mathbf A\cap\mathbb R^n\) over \(\mathbb R\).
\end{proof}

We now exploit strict interlacing to describe the nonnegative and strictly positive solutions of
\eqref{eq:vandermondePOPUC}. Fix an initial point $\varphi_1$ and represent all arguments on the interval
$(\varphi_1,\varphi_1+2\pi)$, that is, we keep
\[
\varphi_1<\varphi_2<\cdots<\varphi_m<\varphi_{m+1}=\varphi_1+2\pi
\]
and choose the arguments $\theta_1,\dots,\theta_n$ so that
\[
\varphi_1<\theta_1<\cdots<\theta_n<\varphi_1+2\pi.
\]
In the sine representation \eqref{eq:circuit-solution-POPUC-sine} this is only a choice of arguments for points on
$\mathbb{S}^1$, and it does not affect the sign pattern of the real circuit vectors. Under strict interlacing, the open arcs
$(\varphi_r,\varphi_{r+1})$ partition the nodes $\{\zeta_j\}_{j=1}^n$ into $m$ bands, and for a circuit supported on one index
per band the two sine products in \eqref{eq:circuit-solution-POPUC-sine} alternate in sign in the same way; hence their
product has constant sign on the support and the circuit can be scaled to be entrywise nonnegative.
\begin{eje}[Bands on $\mathbb{S}^1$ under strict interlacing]\label{ex:bands-popuc}
Take $n=7$, $m=3$ and let
\[
Z_n=\{\zeta_j=e^{i\theta_j}\}_{j=1}^7,\quad
Z_m=\{\xi_k=e^{i\varphi_k}\}_{k=1}^3,
\]
with the normalization $\varphi_1<\theta_1<\cdots<\theta_7<\varphi_1+2\pi$. For instance, choose
\[
(\theta_1,\dots,\theta_7)=\Bigl(\frac{\pi}{6},\frac{\pi}{3},\frac{\pi}{2},\frac{2\pi}{3},\frac{5\pi}{6},\frac{7\pi}{6},\frac{3\pi}{2}\Bigr),
\quad
(\varphi_1,\varphi_2,\varphi_3)=\Bigl(\frac{\pi}{12},\frac{7\pi}{12},\frac{5\pi}{4}\Bigr),
\] \vspace{3pt}
so $\varphi_4=\varphi_1+2\pi=\dfrac{25\pi}{12}$. Then $Z_m$ strictly interlaces $Z_n$ and we may take
\[
i_1=1,\quad i_2=4,\quad i_3=6,
\]
since
\[
\theta_{i_1}<\varphi_1<\theta_{i_1+1},\quad
\theta_{i_2}<\varphi_2<\theta_{i_2+1},\quad
\theta_{i_3}<\varphi_3<\theta_{i_3+1}.
\]
Equivalently, the bands determined by the arcs $(\varphi_r,\varphi_{r+1})$ are
\[
I_1=\{j:\ \varphi_1<\theta_j<\varphi_2\}=\{1,2,3\},\quad
I_2=\{j:\ \varphi_2<\theta_j<\varphi_3\}=\{4,5\},
\]
\[
I_3=\{j:\ \varphi_3<\theta_j<\varphi_4\}=\{6,7\},
\]
and $\{1,\dots,7\}=I_1\bigsqcup I_2\bigsqcup I_3$. If $\mathcal{J}\subset\{1,\dots,7\}$ contains exactly one index from each
band, the denominator in \eqref{eq:circuit-solution-POPUC-sine} has constant sign on $\mathcal{J}$, so the global sign of the circuit vector $\boldsymbol{\omega}^{(\mathcal{J})}$ can be chosen so that $\boldsymbol{\omega}^{(\mathcal{J})}\ge0$.
\end{eje}

We leverage this interlacing structure to characterize the strictly positive solutions of the Vandermonde system.

\begin{lemma}\label{lem:positive-solutions-POPUC}
Assume the hypotheses and notation of Lemma~\ref{lem:circuit-solutions-POPUC}. Assume, in addition, that $Z_m$ strictly
interlaces $Z_n$ in the sense of Definition~\ref{def:strict-interlacing-POPUC}. Fix arguments
$0 \le \varphi_1<\cdots<\varphi_m<\varphi_1+2\pi$ and represent all $\theta_j$ on $(\varphi_1,\varphi_1+2\pi)$ so that
\[
\varphi_1<\theta_1<\cdots<\theta_n<\varphi_1+2\pi,
\quad
\varphi_{m+1}=\varphi_1+2\pi.
\]
Define the bands
\[
I_r=\{j\in\{1,\dots,n\}:\ \varphi_r<\theta_j<\varphi_{r+1}\},\quad r=1,\dots,m,
\]
so that $\{1,\dots,n\}=\bigsqcup_{r=1}^m I_r$. Define
\[
\mathfrak{J}_{+}
=
\Bigl\{\mathcal{J}\subset\{1,\dots,n\}:\ |\mathcal{J}|=m,\ |\mathcal{J}\cap I_r|=1\ \text{for all }r=1,\dots,m\Bigr\}.
\]
For each $\mathcal{J}\in\mathfrak{J}_{+}$, let $\boldsymbol{\omega}^{(\mathcal{J})}$ be the circuit vector given by
\eqref{eq:circuit-solution-POPUC-sine}. Then, $\boldsymbol{\omega}^{(\mathcal{J})}\ge0$.
Moreover, the following hold:
\begin{enumerate}
\item[(a)] Define
\[
\mathcal{W}_{\ge 0}(Z_n,Z_m)
=
\Bigl\{\boldsymbol{\omega}\in\mathbb{R}^n:\mathbf{A}\boldsymbol{\omega}=\mathbf{0}
\ \text{and}\ \omega_j\ge 0\ \text{for all }j\Bigr\}.
\]
Then
\[
\mathcal{W}_{\ge 0}(Z_n,Z_m)
=
\operatorname{cone}\bigl\{\boldsymbol{\omega}^{(\mathcal{J})}:\ \mathcal{J}\in\mathfrak{J}_{+}\bigr\}.
\]
\item[(b)] One has
\[
\mathcal{W}_{>0}(Z_n,Z_m)
=
\mathcal{W}_{\ge0}(Z_n,Z_m)\cap \mathbb{R}^n_{>0}
=
\operatorname{relint}\bigl(\mathcal{W}_{\ge 0}(Z_n,Z_m)\bigr),
\]
where $\operatorname{relint}$ is taken in $\ker\mathbf{A}\cap\mathbb{R}^n$. Furthermore,
$\boldsymbol{\omega}\in\mathcal{W}_{>0}(Z_n,Z_m)$ if and only if
\[
\boldsymbol{\omega}=\sum_{\mathcal{J}\in\mathfrak{J}_{+}} t_{\mathcal{J}}\,\boldsymbol{\omega}^{(\mathcal{J})},
\quad t_{\mathcal{J}}\ge0,
\]
and, for every $j\in\{1,\dots,n\}$, there exists at least one $\mathcal{J}\in\mathfrak{J}_{+}$ such that
$j\in \mathcal{J}$ and $t_{\mathcal{J}}>0$. In particular, $\mathcal{W}_{>0}(Z_n,Z_m)\neq\emptyset$.
\end{enumerate}
\end{lemma}
\begin{proof}
Fix $\mathcal{J}\in\mathfrak{J}_{+}$ and write $\mathcal{J}=\{j_1,\dots,j_m\}$ with $j_r\in I_r$, so that
\[
\varphi_r<\theta_{j_r}<\varphi_{r+1},\quad r=1,\dots,m,
\quad
\theta_{j_1}<\cdots<\theta_{j_m}.
\]
By the chosen normalization, all angles lie in an interval of length $2\pi$, hence
$\theta_{j_r}-\varphi_k\in(-2\pi,2\pi)$ and $\theta_{j_r}-\theta_{j_s}\in(-2\pi,2\pi)$ for all $r,s,k$.
Since $(\theta_{j_r}-\varphi_k)/2,(\theta_{j_r}-\theta_{j_s})/2\in(-\pi,\pi)$ and none of these differences is $0$,
we have
\[
\operatorname{sgn}\!\left(\sin\left(\dfrac{\theta_{j_r}-\varphi_k}{2}\right)\right)=\operatorname{sgn}(\theta_{j_r}-\varphi_k),
\quad
\operatorname{sgn}\!\left(\sin\left(\dfrac{\theta_{j_r}-\theta_{j_s}}{2}\right)\right)=\operatorname{sgn}(\theta_{j_r}-\theta_{j_s}).
\]
Defining
\[
P_m(\theta)=\prod_{k=1}^m(\theta-\varphi_k),
\quad
Q_{\mathcal{J}}(\theta)=\prod_{s=1}^m(\theta-\theta_{j_s}),
\]
it follows from \eqref{eq:circuit-solution-POPUC-sine} that, for each $r=1,\dots,m$,
\[
\operatorname{sgn}\!\left(\omega^{(\mathcal{J})}_{j_r}\right)
=
\operatorname{sgn}\!\left(P_m(\theta_{j_r})\,Q_{\mathcal{J}}'(\theta_{j_r})\right).
\]
Now $\theta_{j_r}\in(\varphi_r,\varphi_{r+1})$ implies that exactly $m-r$ factors in $P_m(\theta_{j_r})$ are negative, hence
\[
\operatorname{sgn}\!\left(P_m(\theta_{j_r})\right)=(-1)^{m-r}.
\]
Moreover,
\[
Q_{\mathcal{J}}'(\theta_{j_r})=\prod_{\substack{s=1\\ s\neq r}}^{m}(\theta_{j_r}-\theta_{j_s}),
\]
and since $\theta_{j_1}<\cdots<\theta_{j_m}$, exactly $m-r$ factors in this product are negative (those with $s>r$), so
\[
\operatorname{sgn}\!\left(Q_{\mathcal{J}}'(\theta_{j_r})\right)=(-1)^{m-r}.
\]
Therefore
\[
\operatorname{sgn}\!\left(P_m(\theta_{j_r})\,Q_{\mathcal{J}}'(\theta_{j_r})\right)=+1
\quad\text{for all }r,
\]
so the nonzero entries of $\boldsymbol{\omega}^{(\mathcal{J})}$ all are positive, as claimed.

The proofs of (a) and (b) are identical to the OPRL case (Lemma~\ref{lem:positive-cone}) once one observes that
$\ker\mathbf{A}\cap\mathbb{R}^n$ is a real linear subspace and $\mathcal{W}_{\ge0}(Z_n,Z_m)$ is a polyhedral cone in that
subspace, with extreme rays generated by the minimal-support nonnegative circuit vectors singled out above.
\end{proof}

We can now combine the quadrature reformulation (Lemma~\ref{lemma:POPUClemma}),
the description of \(\ker\mathbf{A}\cap\mathbb R^n\) (Lemma~\ref{lem:circuit-solutions-POPUC}),
and the positivity analysis under strict interlacing (Lemma~\ref{lem:positive-solutions-POPUC})
to solve Problem~\ref{prob:underdetermined-POPUC}.

\begin{theorem}\label{thm:POPUC-solution}
Assume the notation and hypotheses of Problem~\ref{prob:underdetermined-POPUC}.
Then it admits a solution if and only if \(Z_m\) strictly interlaces \(Z_n\) in the sense of
\ref{def:strict-interlacing-POPUC}. In this case the solution is not unique. More precisely, fix any \(\boldsymbol{\omega}\in\mathcal{W}_{>0}(Z_n,Z_m)\) as in
Lemma~\ref{lem:positive-solutions-POPUC} and consider the discrete measure
\[
\nu=\sum_{j=1}^n \omega_j\,\delta_{\zeta_j}.
\]
Then, the truncated trigonometric moments
\[
\mu_k=\int z^k\,d\nu(z)=\sum_{j=1}^n \omega_j\,\zeta_j^{\,k},
\quad k=-(n-1),\dots,n-1,
\]
determine the Verblunsky coefficients $\alpha_0,\dots,\alpha_{n-2}$.
In addition, the unimodular parameters $b_n$ and $b_m$ are fixed by the prescribed zero sets through the identities in
Lemma~\ref{lemma:POPUClemma}. Hence,
\[
\sigma\bigl(\mathcal{C}(\alpha_0,\dots,\alpha_{n-2},b_n)\bigr)=Z_n,\quad
\sigma\bigl(\mathcal{C}(\alpha_0,\dots,\alpha_{m-2},b_m)\bigr)=Z_m.
\]
Conversely, every solution of Problem~\ref{prob:underdetermined-POPUC} arises from an specific choice of weights
\(\boldsymbol{\omega}\in\mathcal{W}_{>0}(Z_n,Z_m)\) through the quadrature construction in Lemma~\ref{lemma:POPUClemma}
and the positivity mechanism in Lemma~\ref{lem:positive-solutions-POPUC}.
\end{theorem}
\begin{proof}
Necessity follows from the standard POPUC interlacing theorem on $\mathbb S^1$ (see Lemma~\ref{lemma:POPUClemma}), which forces $Z_m$ to strictly interlace $Z_n$ whenever a solution exists.  
Conversely, if $Z_m$ strictly interlaces $Z_n$, then Lemma~\ref{lem:positive-solutions-POPUC} gives $\mathcal{W}_{>0}(Z_n,Z_m)\neq\emptyset$, hence there exists $\boldsymbol{\omega}>0$ solving \eqref{eq:vandermondePOPUC}.  
Given such $\boldsymbol{\omega}$, the quadrature construction in Lemma~\ref{lemma:POPUClemma} produces the truncated moments $$\mu_k=\sum_{j=1}^n\omega_j\zeta_j^{\,k}, \quad k=-(n-1),\dots,n-1,$$ and therefore, Verblunsky coefficients $\alpha_0, \dots, \alpha_{n-2}$ together with a $\mathcal{C}$ matrix with spectrum $Z_n$ and the corresponding truncation with spectrum $Z_m$, with $b_n,b_m$ fixed by the prescribed zeros as in Lemma~\ref{lemma:POPUClemma}.  
Finally, the converse parametrization is immediate from Lemma~2.2, since every solution determines the same truncated moments and hence arises from some choice of $\boldsymbol{\omega}\in\mathcal{W}_{>0}(Z_n,Z_m)$ via the same quadrature procedure.
\end{proof}

As a consequence of Theorem~\ref{thm:POPUC-solution}, we obtain the following unit-circle analogue of the generalized
Wendroff statement.

\begin{coro}[Generalized Wendroff theorem on $\mathbb{S}^1$]\label{cor:gen-Wendroff-POPUC}
Fix integers \(1\le m<n\) and let
\[
Z_n=\{\zeta_j\}_{j=1}^{n},\quad Z_m=\{\xi_k\}_{k=1}^{m}
\]
be sets of distinct points on \(\mathbb S^1\). Then, there exists a positive Borel measure \(\nu\) on \(\mathbb S^1\), and parameters \(b_n,b_m\in\mathbb S^1\) such that the associated POPUC satisfy
\[
\text{zeros}\left(\Psi_n(z;\nu,b_n)\right)=Z_n,\quad
\text{zeros}\left(\Psi_m(z;\nu,b_m)\right)=Z_m,
\]
if and only if \(Z_m\) strictly interlaces \(Z_n\) in the sense of \ref{def:strict-interlacing-POPUC}.
\end{coro}

\begin{proof}
{  This follows from Theorem~\ref{thm:POPUC-solution} and Lemma~\ref{lemma:POPUClemma}, together with Verblunsky's theorem \cite{S05I}.}
\end{proof}

\section{Numerical algorithms and examples} \label{sec:algos}
Now we can present Algorithm~\ref{alg:oprl-reconstruction}, which formalises the reconstruction procedure underlying the results above. Among the steps in the above procedure, Step~4 is the only genuinely nontrivial one.
Indeed, constructing all strictly positive weights amounts to selecting the admissible conical
combinations of the nonnegative circuit vectors from \eqref{eq:circuit-solution} supported on
$\mathfrak{J}_{+}$ (Lemma~\ref{lem:positive-cone}), and the resulting combinatorics can vary
substantially with the fine interlacing pattern (in particular, with the band decomposition determined
by the indices from Definition~\ref{def:strict-interlacing-oprl}). For the small instance below, the
admissible combinations can be described explicitly and the construction is completely transparent.

\begin{algorithm}[t]
\caption{Reconstruction in the OPRL setting}
\label{alg:oprl-reconstruction}
\begin{tcolorbox}[algobox]
\begin{algorithmic}[1]
\Require Sets $Z_n=\{x_j\}_{j=1}^n$ and $Z_m=\{y_j\}^m_{j=1}$ as in Problem~\ref{prob:underdetermined-oprl}, with $Z_m$ strictly interlacing $Z_n$ in the sense of Definition~\ref{def:strict-interlacing-oprl}.
\Ensure A family $\{P_k\}_{k=0}^n$ and an associated Jacobi matrix $J_n$ consistent with $(Z_n,Z_m)$.

\State Determine the interlacing indices
\[
i_0=0<i_1<\cdots<i_m<n=i_{m+1},
\]
as in Definition~\ref{def:strict-interlacing-oprl}.

\State Define the bands $I_r$ as in Lemma~\ref{lem:positive-cone} and compute the index set \(\mathfrak{J}_{+}\).

\State For each \(\mathcal{J}\in\mathfrak{J}_{+}\), construct the nonnegative vector \(\boldsymbol{\omega}^{(\mathcal{J})}\) via \eqref{eq:circuit-solution}.

\State Parametrise all \(\boldsymbol{\omega}\in\mathcal{W}_{>0}(Z_n,Z_m)\) as adequate conical combinations of the vectors \(\boldsymbol{\omega}^{(\mathcal{J})}\), as in Lemma~\ref{lem:positive-cone}(b).

\State Form the quadrature measure \(\mu=\sum_{j=1}^n \omega_j\,\delta_{x_j}\) and compute the moments
\[
\mu_k=\sum_{j=1}^n \omega_j x_j^k,\quad k=0,\ldots,2n-1.
\]

\State Apply the Stieltjes procedure to recover the recurrence coefficients \(\{\beta_k\}\) and \(\{\gamma_k\}\), compute \(\{P_k\}_{k=0}^{n}\), and assemble the associated Jacobi matrix \(J_n\).
\end{algorithmic}
\end{tcolorbox}
\end{algorithm}

\begin{eje}[Running the algorithm for $n=4$, $m=2$]
\label{ex:oprl}
Let
\[
Z_n=\{x_1,x_2,x_3,x_4\}=\{1,2,3,4\},\quad
Z_m=\{y_1,y_2\}=\Bigl\{\frac{3}{2},\frac{7}{2}\Bigr\}.
\]
Then, $Z_m$ strictly interlaces $Z_n$ in the sense of Definition~\ref{def:strict-interlacing-oprl} (see Figure~\ref{fig:oprl-example-bw}), with
\[
i_0=0<i_1=1<i_2=3<i_3=4.
\]
The associated bands (Lemma~\ref{lem:positive-cone}) are
\[
I_0=\{1\},\quad I_1=\{2,3\},\quad I_2=\{4\},
\]
and therefore
\[
\mathfrak{J}_{+}=\bigl\{\{1,2,4\},\{1,3,4\}\bigr\}.
\]
Using \eqref{eq:circuit-solution}, the two nonnegative circuit solutions are
\[
\boldsymbol{\omega}^{(1)}=\left(\frac{4}{15},\frac{2}{3},0,\frac{2}{15}\right)^{\mathsf T},\quad
\boldsymbol{\omega}^{(2)}=\left(\frac{2}{15},0,\frac{2}{3},\frac{4}{15}\right)^{\mathsf T}.
\]
Hence, by Lemma~\ref{lem:positive-cone}(b), every strictly positive solution can be written as
\[
\boldsymbol{\omega}=\boldsymbol{\omega}^{(1)}+s_1\,\boldsymbol{\omega}^{(2)},\quad s_1>0,
\]
that is,
\[
\boldsymbol{\omega}=\left(\frac{2s_1}{15}+\frac{4}{15},\frac{2}{3},\frac{2s_1}{3},\frac{4s_1}{15}+\frac{2}{15}\right)^{\mathsf T}.
\]
The quadrature measure is $\mu=\sum_{j=1}^4 \omega_j\,\delta_{x_j}$ and the moments are computed by
\[
\mu_k=\sum_{j=1}^4 \omega_j x_j^k,\quad k=0,\dots,7.
\]
Applying the Stieltjes procedure yields the following
recurrence coefficients
\begin{align*}
\beta_0&=\frac{3s_1+2}{s_1+1},\quad
\beta_1=\frac{2s_1+3}{s_1+1},\quad
\beta_2=\frac{3s_1+2}{s_1+1},\quad
\beta_3=\frac{2s_1+3}{s_1+1},\\[7pt]
\gamma_1&=\frac{3s_1^2+10s_1+3}{4(s_1+1)^2},\quad
\gamma_2=\frac{15(s_1+1)^2}{4(3s_1^2+10s_1+3)},\\[7pt]
\gamma_3&=\frac{4s_1(2s_1^2+5s_1+2)}{(s_1+1)(3s_1^3+13s_1^2+13s_1+3)}. \\[3pt]
\end{align*}
The corresponding monic OPRL begin as
\[
P_0(x)=1,\quad
P_1(x)=x-\frac{3s_1+2}{s_1+1},\quad
P_2(x)=x^2-5x+\frac{21}{4},
\]
and $P_4(x)=(x-1)(x-2)(x-3)(x-4)$. We omit $P_3$, see Appendix \ref{app:wolfram-oprl}. Finally, the associated Jacobi matrix $J_4$ is
\[
J_4=
\begin{bmatrix}
\;\beta_0 & 1 & 0 & 0\;\\[7pt]
\;\gamma_1 & \beta_1 & 1 & 0\;\\[7pt]
\;0 & \gamma_2 & \beta_2 & 1\;\\[7pt]
\;0 & 0 & \gamma_3 & \beta_3\;
\end{bmatrix}.
\]
This matrix satisfies
\[
\sigma(J_4)=Z_4,\quad \sigma(J_2)=Z_2.
\]

\end{eje}

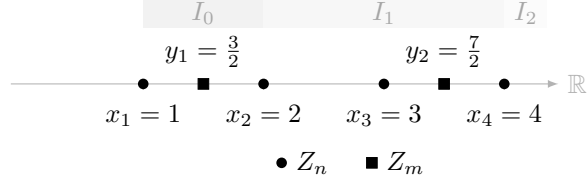
\begin{figure}[H]
\centering
\begin{tikzpicture}[x=1.6cm,y=1cm,>=latex]
  \draw[->,gray!60] (-0.1,0)--(4.45,0) node[right] {$\mathbb{R}$};

  \fill[gray!10] (1,0.75) rectangle (2,1.15);
  \fill[gray!6]  (2,0.75) rectangle (4,1.15);
  \fill[gray!3]  (4,0.75) rectangle (4.35,1.15);

  \node[gray!70] at (1.5,0.95) {$I_0$};
  \node[gray!70] at (3.0,0.95) {$I_1$};
  \node[gray!70] at (4.18,0.95) {$I_2$};

  \foreach \x/\lab in {1/{x_1=1},2/{x_2=2},3/{x_3=3},4/{x_4=4}} {
    \draw[gray!60] (\x,0) -- (\x,-0.08);
    \node[circle,fill=black,inner sep=0pt,minimum size=4pt] at (\x,0) {};
    \node[below] at (\x,-0.16) {$\lab$};
  }

  \node[draw,fill=black,inner sep=0pt,minimum size=4pt] at (1.5,0) {};
  \node[above] at (1.5,0.10) {$y_1=\tfrac{3}{2}$};

  \node[draw,fill=black,inner sep=0pt,minimum size=4pt] at (3.5,0) {};
  \node[above] at (3.5,0.10) {$y_2=\tfrac{7}{2}$};

  \begin{scope}[shift={(2.15,-1.05)}]
    \node[circle,fill=black,inner sep=0pt,minimum size=4pt] at (0,0) {};
    \node[right,xshift=2pt] at (0,0) {$Z_n$};
    \node[draw,fill=black,inner sep=0pt,minimum size=4pt] at (0.75,0) {};
    \node[right,xshift=2pt] at (0.75,0) {$Z_m$};
  \end{scope}
\end{tikzpicture}
\caption{The sets \(Z_n=\{1,2,3,4\}\) (disks) and \(Z_m=\{\tfrac{3}{2},\tfrac{7}{2}\}\) (squares), with bands \(I_0=\{1\}\), \(I_1=\{2,3\}\), \(I_2=\{4\}\).}
\label{fig:oprl-example-bw}
\end{figure}

We have also implemented the same steps for larger values of $n$ and $m$ and for families with
several free parameters. In those cases, while the numerical construction remains stable, the closed
forms for the resulting recurrence coefficients $\beta_k$ and $\gamma_k$ quickly become unwieldy.
For this reason, and so that the reader may verify the example by hand, in the traditional pencil-and-paper manner, we restrict ourselves to the small illustrative example above. In the Appendix~\ref{app:wolfram-oprl} we provide the \textsc{Mathematica} code for this algorithm, so that the reader may experiment with substantially more complex examples without the computation becoming prohibitively time-consuming.

For the POPUC setting, the algorithm is completely analogous: one replaces the OPRL cone of admissible weights by
\(\mathcal{W}_{>0}(Z_n,Z_m)\) from Section~\ref{sec:POPUC}, forms the discrete measure
\(\nu=\sum_{j=1}^n \omega_j\,\delta_{\zeta_j}\), computes the truncated trigonometric moments
\(\mu_k=\sum_{j=1}^n \omega_j\,\zeta_j^{\,k}\) for \(k=-(n-1),\dots,n-1\), and then recovers the
Verblunsky coefficients \(\alpha_0,\dots,\alpha_{n-2}\) by the Szeg\H{o} recursion. 

\begin{algorithm}[t]
\caption{Reconstruction in the POPUC setting}
\label{alg:popuc-reconstruction}
\begin{tcolorbox}[algobox]
\begin{algorithmic}[1]
\Require Sets \(Z_n=\{\zeta_j\}_{j=1}^n\subset\mathbb S^1\) and \(Z_m=\{\xi_k\}_{k=1}^m\subset\mathbb S^1\) as in Problem~\ref{prob:underdetermined-POPUC}, with \(Z_m\) strictly interlacing \(Z_n\) in the sense of Definition~\ref{def:strict-interlacing-POPUC}.
\Ensure A POPUC family \(\{\Psi_k\}_{k=1}^n\) and compatible unitary pentadiagonal matrices \(\mathcal C_n\) and \(\mathcal C_m\) realising \((Z_n,Z_m)\).

\State Fix arguments \(\zeta_j=e^{i\theta_j}\), \(\xi_k=e^{i\varphi_k}\) with
\[
\varphi_1<\theta_1<\cdots<\theta_n<\varphi_1+2\pi,\quad
\varphi_1<\varphi_2<\cdots<\varphi_m<\varphi_{m+1}=\varphi_1+2\pi.
\]

\State Define the bands
\[
I_r=\{\,j\in\{1,\dots,n\}:\ \varphi_r<\theta_j<\varphi_{r+1}\,\},\quad r=1,\dots,m,
\]
and form the admissible index family
\[
\mathfrak J_{+}=\Bigl\{\mathcal J\subset\{1,\dots,n\}:\ |\mathcal J|=m,\ |\mathcal J\cap I_r|=1\ \text{for all }r=1,\dots,m\Bigr\}.
\]

\State For each \(\mathcal J\in\mathfrak J_{+}\), construct the nonnegative circuit vector \(\boldsymbol\omega^{(\mathcal J)}\in\mathbb R^n\) via \eqref{eq:circuit-solution-POPUC-sine}.

\State Parametrise all \(\boldsymbol\omega\in\mathcal W_{>0}(Z_n,Z_m)\) as suitable conical combinations of the vectors \(\boldsymbol\omega^{(\mathcal J)}\), as in Lemma~\ref{lem:positive-solutions-POPUC}(b), ensuring strict positivity in every coordinate.

\State Form the discrete measure
\[
\nu=\sum_{j=1}^n \omega_j\,\delta_{\zeta_j},
\]
and compute the truncated trigonometric moments
\[
\mu_k=\int z^k\,d\nu(z)=\sum_{j=1}^n \omega_j\,\zeta_j^{\,k},\quad k=-(n-1),\dots,n-1,
\quad \mu_{-k}=\overline{\mu_k}.
\]

\State Apply the Schur algorithm to recover the Verblunsky coefficients \(\alpha_0,\dots,\alpha_{n-2}\).

\State Set the boundary parameters from the prescribed zero sets (Lemma~\ref{lemma:POPUClemma}):
\[
b_n=-\overline{\Psi_n(0)},\quad b_m=-\overline{\Psi_m(0)},
\]
where \(\Psi_n(z)=\prod_{j=1}^n (z-\zeta_j)\) and \(\Psi_m(z)=\prod_{k=1}^m (z-\xi_k)\).

\State Assemble the unitary pentadiagonal matrices
\[
\mathcal C_n=\mathcal C(\alpha_0,\dots,\alpha_{n-2},b_n),\quad
\mathcal C_m=\mathcal C(\alpha_0,\dots,\alpha_{m-2},b_m),
\]
and, if desired, recover \(\Psi_k(z)=\det(zI_k-\mathcal C_k)\) (up to a unimodular constant) for \(k=1,\dots,n\).

\end{algorithmic}
\end{tcolorbox}
\end{algorithm}

\begin{eje}[Running the construction for \(n=3\), \(m=2\)]
\label{ex:popuc}
Let
\[
Z_n=\{\zeta_1,\zeta_2,\zeta_3\}=\Bigl\{i,\,e^{\frac{4\pi i}{3}},\,e^{\frac{5\pi i}{3}}\Bigr\},\quad
Z_m=\{\xi_1,\xi_2\}=\{1,-1\}.
\]
Write \(\zeta_j=e^{i\theta_j}\) and \(\xi_k=e^{i\varphi_k}\) with
\[
(\theta_1,\theta_2,\theta_3)=\Bigl(\frac{\pi}{2},\frac{4\pi}{3},\frac{5\pi}{3}\Bigr),\quad
(\varphi_1,\varphi_2)=(0,\pi).
\]
Fix the base point \(\varphi_1\) and view all arguments in the interval \((\varphi_1,\varphi_1+2\pi)=(0,2\pi)\).
Then \(Z_m\) strictly interlaces \(Z_n\) (see Figure~\ref{fig:popuc-example-bw}), and the bands between consecutive points of \(Z_m\) are
\[
I_1=\{j:\ 0<\theta_j<\pi\}=\{1\},\quad
I_2=\{j:\ \pi<\theta_j<2\pi\}=\{2,3\}.
\]
Hence
\[
\mathfrak{J}_{+}=\bigl\{\{1,2\},\{1,3\}\bigr\}.
\]
Using \eqref{eq:circuit-solution-POPUC-sine}, the
corresponding nonnegative circuit solutions are
\[
\boldsymbol{\omega}^{(1)}=\Bigl(2(\sqrt6-\sqrt2),\,\frac{4}{3}(3\sqrt2-\sqrt6),\,0\Bigr)^{\mathsf T},\quad
\boldsymbol{\omega}^{(2)}=\Bigl(2(\sqrt6-\sqrt2),\,0,\,\frac{4}{3}(3\sqrt2-\sqrt6)\Bigr)^{\mathsf T}.
\]
Therefore every strictly positive solution is obtained by a conical combination covering all indices, for instance
\[
\boldsymbol{\omega}=\boldsymbol{\omega}^{(1)}+s_1\,\boldsymbol{\omega}^{(2)},\quad s_1>0,
\]
that is,
\[
\boldsymbol{\omega}=\Bigl(2(\sqrt6-\sqrt2)(1+s_1),\,\frac{4}{3}(3\sqrt2-\sqrt6),\,\frac{4s_1}{3}(3\sqrt2-\sqrt6)\Bigr)^{\mathsf T}.
\]
 For concreteness, we fix $s_1=1$ and set \(\nu=\sum_{j=1}^3\omega_j\,\delta_{\zeta_j}\). The truncated trigonometric moments are
\[
\mu_k=\int z^k\,d\nu(z)=\sum_{j=1}^3\omega_j\,\zeta_j^{\,k},\quad k=-2,-1,0,1,2,
\quad \mu_{-k}=\overline{\mu_k}.
\]
A direct computation gives
\[
\mu_0=4\sqrt2+\frac{4\sqrt6}{3},\quad \mu_1=0,\quad \mu_2=-\frac{8\sqrt6}{3}.
\]
Applying the Schur algorithm yields
\[
\alpha_0=\frac{\overline{\mu_1}}{\mu_0}=0,\quad
\alpha_1=\frac{\overline{\mu_2}}{\mu_0}=1-\sqrt3,\quad
\rho_1=\sqrt{1-|\alpha_1|^2}=\sqrt{2\sqrt3-3}.
\]
Moreover, the unimodular parameters are fixed by the prescribed zeros,
\[
b_3=i,\quad b_2=1.
\]
Finally,
\[
\mathcal{C}_3=\mathcal{C}(\alpha_0,\alpha_1,b_3)=
\begin{bmatrix}
0 & 1-\sqrt3 & \sqrt{2\sqrt3-3}\;\\[7pt]
1 & 0 & 0\\[7pt]
0 & -\sqrt{3-2\sqrt3} & i(1-\sqrt3)
\end{bmatrix},\quad
\mathcal{C}_2=\mathcal{C}(\alpha_0,b_2)=
\begin{bmatrix}
0 & 1\\[7pt]
1 & 0
\end{bmatrix},
\]
satisfy 
\[
\sigma(\mathcal{C}_3)=Z_n,\quad \sigma(\mathcal{C}_2)=Z_m.
\]
\end{eje}

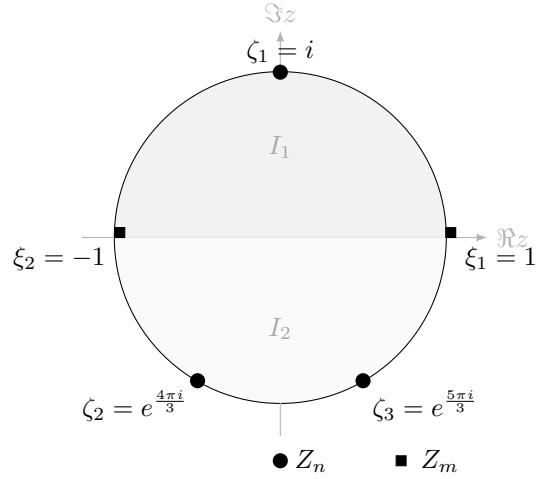
\begin{figure}[H]
\centering
\begin{tikzpicture}[scale=2.2,>=latex]
  \draw[thick] (0,0) circle (1);

  \draw[->,gray!60] (-1.2,0)--(1.25,0) node[right] {$\Re z$};
  \draw[->,gray!60] (0,-1.2)--(0,1.25) node[above] {$\Im z$};

  \fill[gray!10] (1,0) arc[start angle=0,end angle=180,radius=1] -- (0,0) -- cycle;
  \fill[gray!4]  (-1,0) arc[start angle=180,end angle=360,radius=1] -- (0,0) -- cycle;

  \node[gray!70] at (0,0.55) {$I_1$};
  \node[gray!70] at (0,-0.55) {$I_2$};

  \filldraw[black] (1,0)   rectangle ++(0.06,0.06);
  \node[below right] at (1.06,0) {$\xi_1=1$};

  \filldraw[black] (-1,0)  rectangle ++(0.06,0.06);
  \node[below left] at (-1,0) {$\xi_2=-1$};

  \fill[black] (0,1) circle (1.3pt);
  \node[above] at (0,1) {$\zeta_1=i$};

  \fill[black] (-0.5,-0.8660254) circle (1.3pt);
  \node[below left] at (-0.5,-0.8660254) {$\zeta_2=e^{\frac{4\pi i}{3}}$};

  \fill[black] (0.5,-0.8660254) circle (1.3pt);
  \node[below right] at (0.5,-0.8660254) {$\zeta_3=e^{\frac{5\pi i}{3}}$};

  \begin{scope}[shift={(0,-1.35)}]
    \fill[black] (0,0) circle (1.3pt);
    \node[right,xshift=2pt] at (0,0) {$Z_n$};
    \filldraw[black] (0.7,-0.03) rectangle ++(0.06,0.06);
    \node[right,xshift=2pt] at (0.76,0) {$Z_m$};
  \end{scope}
\end{tikzpicture}
\caption{The sets \(Z_n=\{i,e^{4\pi i/3},e^{5\pi i/3}\}\) (disks) and \(Z_m=\{1,-1\}\) (squares) on the unit circle, with bands \(I_1=(0,\pi)\) and \(I_2=(\pi,2\pi)\).}
\label{fig:popuc-example-bw}
\end{figure}
The above construction is straightforward to implement and works well in practice. However, in the POPUC setting the
appearance of complex phases (already at the level of the circuit weights and the trigonometric moments) makes the
symbolic expressions grow rapidly and become less readable than in the OPRL case. 

\section{Final remarks: beyond two spectra}\label{sec:remarks}
Theorems~\ref{thm:OPRL-solution} and~\ref{thm:POPUC-solution} show that the expected strict
interlacing condition is the sharp criterion for solvability of the underdetermined two-spectra
problems considered in this paper. Moreover, whenever \(m<n-1\) the solution is never unique: the
admissible measures (equivalently, the admissible Jacobi/unitary
pentadiagonal matrices) form a family parametrized by
the strictly positive solutions of the corresponding Vandermonde-type system. This nonuniqueness is
already visible in the free parameters appearing in the examples of Section~\ref{sec:algos}. In
general, the farther \(m\) lies from \(n\), and the more uneven the interlacing pattern is across the
bands, the larger the dimension of the feasible cone becomes and the more intricate the resulting
parametrizations of the admissible weights (and hence of the recurrence data).

A natural extension is the multispectra setting: prescribe \(Z_n\) together with several spectra
\(Z_{m_1},Z_{m_2},\dots\) coming from leading principal submatrices of sizes
\(m_1<m_2<\cdots<n\), in either a underdetermined or (at most) fully determined regime. It is
immediate that pairwise strict interlacing is necessary, but it is not sufficient in general. For
instance, when \(m_1=n-1\) the data \((Z_n,Z_{m_1})\) already determine a unique Jacobi matrix, so
one may choose a set \(Z_{m_2}\) that strictly interlaces \(Z_n\) (and also interlaces \(Z_{m_1}\) in
the expected way) but that does not coincide with the spectrum of the actual leading principal
submatrix \(J_{m_2}\) of the uniquely determined matrix. This suggests a genuine compatibility
problem beyond pairwise interlacing.

From the present viewpoint, Lemma~\ref{lemma:OPRLlemma} (and its POPUC analogue) can be adapted with
no essential changes: one obtains a single linear system whose coefficient matrix has several
Vandermonde-type blocks, each weighted by the corresponding polynomial \(P_{m_i}\). Interlacing
alone does not resolve positivity of the solutions in this higher-block system, but solvability can
still be phrased as the existence of a strictly positive solution of the resulting Vandermonde
system. This opens the door to a mixed analytic--computational approach: solve the block system
explicitly (or parametrically), then impose positivity by standard analytic criteria. While the
necessary hypotheses and a complete structural description remain to be developed, the numerical
algorithm of Section~\ref{sec:algos} extends directly to this setting by treating the linear system
with multiple blocks and searching for positive solutions. For the sake of conciseness, we omit
examples in the multispectra case.

Finally, it would be very interesting to complement the above feasibility reformulation with
\emph{intrinsic} criteria for the multispectra regime, closer in spirit to the interlacing tests in the
two--spectra case. In other words, one would like necessary and sufficient conditions that encode
compatibility across levels \(m_1<m_2<\cdots<n\) in a genuinely analytic way (e.g.\ inequalities or
structured identities), rather than via solving the multi-block Vandermonde system and then checking
positivity. We regard this as a natural and nontrivial extension of the present work, and we hope
that the formulation in terms of block systems and positive solutions may serve as a useful starting
point for such a characterization.

\section*{Acknowledgements}
The authors acknowledge financial support from the Centre for Mathematics of the University of Coimbra (CMUC), funded by the Portuguese Foundation for Science and Technology (FCT), under the projects UID/00324/2025 (\url{https://doi.org/10.54499/UID/00324/2025}) and UID/PRR/00324/2025. The first author acknowledges financial support from the FCT under the grant \url{https://doi.org/10.54499/2022.00143.CEECIND/CP1714/CT0002}. The second author acknowledges financial support from FCT under the grant \href{https://doi.org/10.54499/UIDB.154694.2023}{DOI: 10.54499/UIDB.154694.2023}.

\bibliographystyle{amsplain}
\bibliography{bib}

\appendix
\section{Mathematica code}\label{app:wolfram-oprl}

As an illustration, we provide below a \textsc{Mathematica} implementation of Algorithm~\ref{alg:oprl-reconstruction}. The reader may proceed analogously, when needed, with Algorithm~\ref{alg:popuc-reconstruction}.

\lstset{style=wlcode}
\begin{lstlisting}[caption={Wolfram Language code for the OPRL reconstruction.},label={lst:wolfram-oprl}]
ClearAll["Global`*"];
t = Symbol["t"];

AssertTrue[cond_, msg_String] :=
  If[TrueQ[cond], Null, Print["[ERROR] ", msg]; Abort[]];

MonicFromZeros[Z_List] := Expand[Times @@ (t - # & /@ Z)];

BandsFromInterlacingRL[ZnSorted_List, ZmSorted_List] :=
 Module[{n, m, bands, j},
  n = Length[ZnSorted];
  m = Length[ZmSorted];
  bands = Table[{}, {m + 1}];
  j = 1;
  Do[
   While[j <= m && ZnSorted[[i]] > ZmSorted[[j]], j++];
   bands[[j]] = Append[bands[[j]], i];
   ,
   {i, 1, n}
  ];
  bands
];

MonicPolyFromZerosRL[zeros_List] := MonicFromZeros[zeros];

OmegaPrimeAtNodeRL[Zn_List, J_List, i_Integer] :=
 Times @@ (Zn[[i]] - Zn[[#]] & /@ DeleteCases[J, i]);

LambdaFromIndicesRL[Zn_List, Pm_, J_List] :=
 Module[{n, lam, denom, sgn},
  n = Length[Zn];
  lam = ConstantArray[0, n];
  Do[
   denom = (Pm /. t -> Zn[[i]])*OmegaPrimeAtNodeRL[Zn, J, i];
   lam[[i]] = 1/denom;
   ,
   {i, J}
  ];
  sgn = Sign[Total[Sign /@ N[lam[[J]], 30]]];
  If[sgn < 0, lam = -lam];
  lam
];

InnerRL[p_, q_, Zn_List, w_List] :=
 Total[w*(p /. t -> Zn)*(q /. t -> Zn)];

StieltjesMonicOPRL[Zn_List, w_List] :=
 Module[{n, P, h, beta, gamma, k, Pk, hk, xPk, betaK, gammaK, Pnext},
  n = Length[Zn];
  P = ConstantArray[0, n + 1];
  h = ConstantArray[0, n + 1];
  beta = ConstantArray[0, n];
  gamma = ConstantArray[0, n];

  P[[1]] = 1;
  h[[1]] = InnerRL[P[[1]], P[[1]], Zn, w];

  Do[
   Pk = P[[k + 1]];
   hk = h[[k + 1]];
   xPk = Expand[t*Pk];

   betaK = InnerRL[xPk, Pk, Zn, w]/hk;
   beta[[k + 1]] = betaK;

   If[k == 0,
    Pnext = Expand[xPk - betaK*Pk],
    gammaK = h[[k + 1]]/h[[k]];
    gamma[[k + 1]] = gammaK;
    Pnext = Expand[xPk - betaK*Pk - gammaK*P[[k]]]
   ];

   P[[k + 2]] = Pnext;
   h[[k + 2]] = InnerRL[Pnext, Pnext, Zn, w];
   ,
   {k, 0, n - 1}
  ];

  P
];

OPRLFamily[ZnIn_List, ZmIn_List, specRules_ : Automatic] :=
 Module[{Zn, Zm, n, m, bands, Jplus, Pm, lambdaVectors, K, params,
   rules, wSym, wUse, PList},

  Zn = Sort@ZnIn;
  Zm = Sort@ZmIn;
  n = Length[Zn];
  m = Length[Zm];

  AssertTrue[1 <= m < n, "Require 1 <= m < n."];
  AssertTrue[And @@ Thread[Zn =!= 0], "Assume Zn has no zero (optional)."];

  bands = BandsFromInterlacingRL[Zn, Zm];
  AssertTrue[And @@ (Length[#] >= 1 & /@ bands), "Some band is empty."];

  Jplus = Tuples[bands];
  Pm = MonicPolyFromZerosRL[Zm];

  lambdaVectors = LambdaFromIndicesRL[Zn, Pm, #] & /@ Jplus;
  K = Length[lambdaVectors];

  If[specRules === Automatic,
   wUse = Total[lambdaVectors];
   rules = {};
   ,
   params = If[K <= 1, {}, Array[s, K - 1]];
   wSym =
    If[K <= 1,
     lambdaVectors[[1]],
     lambdaVectors[[1]] + Sum[params[[j]] lambdaVectors[[j + 1]], {j, 1, K - 1}]
    ];
   rules = Which[specRules === {}, {}, True, specRules];
   wUse = wSym /. rules;
  ];

  PList = StieltjesMonicOPRL[Zn, wUse];

  Print["Bands I_r = ", bands];
  If[Length[Jplus] <= 20,
   Print["J_+ = ", Jplus],
   Print["J_+ has length ", Length[Jplus], " (showing first 10) = ", Take[Jplus, 10]]
  ];
  If[specRules =!= Automatic && K > 1, Print["w(params) = ", wSym]];
  If[specRules =!= Automatic && rules =!= {}, Print["specialisation = ", rules]];
  Print["w = ", wUse];
  Print["OPRL P_k(t), k = 1,...,", n, ":"];
  Do[Print["P_", k, "(t) = ", PList[[k + 1]]], {k, 1, n}];

  Rest[PList]
];
\end{lstlisting}

The routine \lstinline[style=wlcode]|OPRLFamily| is a direct implementation of the construction described in
Lemmas~\ref{lem:circuit-solutions} and~\ref{lem:positive-cone} (and therefore of the reconstruction mechanism in
Theorem~\ref{thm:OPRL-solution}). Given $(Z_n,Z_m)$, the command
\lstinline[style=wlcode]|BandsFromInterlacingRL| computes the band decomposition
\(
\{1,\dots,n\}=\bigsqcup_{r=0}^{m} I_r
\)
associated with the interlacing indices from Definition~\ref{def:strict-interlacing-oprl}. The subsequent line
\lstinline[style=wlcode]|Jplus = Tuples[bands]| enumerates the admissible index family
\[
\mathfrak{J}_{+}
=
\Bigl\{\mathcal{J}\subset\{1,\dots,n\}:\ |\mathcal{J}|=m+1,\ |\mathcal{J}\cap I_r|=1\ \text{for all }r=0,\dots,m\Bigr\},
\]
as in Lemma~\ref{lem:positive-cone}. For each $\mathcal{J}\in\mathfrak{J}_{+}$, the function
\lstinline[style=wlcode]|LambdaFromIndicesRL| constructs the corresponding sparse circuit solution
\(\boldsymbol{\omega}^{(\mathcal{J})}\) given by \eqref{eq:circuit-solution} (implemented through the factor
\(P_m(x_i)\,Q_{\mathcal{J}}'(x_i)\) in the denominator), and performs the global sign choice from
Lemma~\ref{lem:positive-cone} so that the resulting vector is entrywise nonnegative. The weight vector used for the Stieltjes reconstruction is denoted by \lstinline[style=wlcode]|wUse| in the code; it
corresponds to a choice \(\boldsymbol{\omega}\in\mathcal{W}_{>0}(Z_n,Z_m)\) in the notation of
Section~\ref{sec:OPRL}. If one calls \lstinline[style=wlcode]|OPRLFamily[Zn, Zm]| with only two arguments, then the third
argument \lstinline[style=wlcode]|specRules| takes its default value \lstinline[style=wlcode]|Automatic|. In that case the
routine does \emph{not} introduce any symbolic parameters \lstinline[style=wlcode]|s[1]|, \lstinline[style=wlcode]|s[2]|,
etc. Instead, it selects a concrete strictly positive weight vector by the explicit choice
\[
\boldsymbol{\omega}
=\sum_{\mathcal{J}\in\mathfrak{J}_{+}}\boldsymbol{\omega}^{(\mathcal{J})},
\]
that is, it assigns coefficient $1$ to every nonnegative circuit vector and sums them all. This produces an element of
\(\mathcal{W}_{>0}(Z_n,Z_m)\) (Lemma~\ref{lem:positive-cone}(b)) without displaying any free parameters. The subsequent
Stieltjes step is then carried out with this fixed vector \(\boldsymbol{\omega}\), yielding one admissible OPRL family
\(\{P_k\}_{k=0}^n\) and one associated Jacobi matrix \(J_n\) consistent with \((Z_n,Z_m)\). If instead one calls \lstinline[style=wlcode]|OPRLFamily[Zn, Zm, {}]|, the routine keeps a symbolic parametrization of the
cone
\[
\mathcal{W}_{\ge0}(Z_n,Z_m)
=
\operatorname{cone}\bigl\{\boldsymbol{\omega}^{(\mathcal{J})}:\ \mathcal{J}\in\mathfrak{J}_{+}\bigr\}
\]
by forming a linear combination with free parameters \lstinline[style=wlcode]|s[1]|, \lstinline[style=wlcode]|s[2]|, \dots\ .
Finally, providing a list of replacement rules, e.g.
\lstinline[style=wlcode]|OPRLFamily[Zn, Zm, {s[1] -> 1}]|,
specializes these parameters and yields the corresponding admissible weights and reconstructed family.

For the same data $(Z_n,Z_m)$ we present three representative outputs: the default call
\lstinline[style=wlcode]|OPRLFamily[Zn, Zm]| (which selects a concrete strictly positive weight vector),
the symbolic parametrisation
\lstinline[style=wlcode]|OPRLFamily[Zn, Zm, {}]| (which keeps the free parameter(s) \lstinline[style=wlcode]|s[1]|, \lstinline[style=wlcode]|s[2]|, \dots\ visible),
and a concrete specialisation such as
\lstinline[style=wlcode]|OPRLFamily[Zn, Zm, {s[1] -> 3}]|.

For the small test case \(Z_4=\{1,2,3,4\}\) and \(Z_2=\{\tfrac{3}{2},\tfrac{7}{2}\}\), we run
\lstinline[style=wlcode]|OPRLFamily[{1, 2, 3, 4}, {3/2, 7/2}]|
and record its output below.

\lstset{style=wlcode}
\begin{lstlisting}
OPRLFamily[{1, 2, 3, 4}, {3/2, 7/2}];

Bands I_r = {{1},{2,3},{4}}

J_+ = {{1,2,4},{1,3,4}}

w = {2/5,2/3,2/3,2/5}

OPRL P_k(t), k = 1,...,4:

P_1(t) = -(5/2)+t

P_2(t) = 21/4-5 t+t^2

P_3(t) = -(345/32)+(269 t)/16-(15 t^2)/2+t^3

P_4(t) = 24-50 t+35 t^2-10 t^3+t^4
\end{lstlisting}

For the same data, we next display the symbolic parametrisation produced by
\lstinline[style=wlcode]|OPRLFamily[{1, 2, 3, 4}, {3/2, 7/2}, {}]|.

\lstset{style=wlcode}
\begin{lstlisting}
OPRLFamily[{1, 2, 3, 4}, {3/2, 7/2},{}];

Bands I_r = {{1},{2,3},{4}}

J_+ = {{1,2,4},{1,3,4}}

w(params) = {4/15 + (2 s[1])/15, 2/3, (2 s[1])/3, 2/15 + (4 s[1])/15}

w = {4/15 + (2 s[1])/15, 2/3, (2 s[1])/3, 2/15 + (4 s[1])/15}

OPRL P_k(t), k = 1,...,4:

P_1(t) = -((2 + 3 s[1])/(1 + s[1])) + t

P_2(t) = 21/4 - 5 t + t^2

P_3(t) = 3*(-12 s[1]^3 - 53 s[1]^2 - 42 s[1] - 8)/(3 s[1]^3 + 13 s[1]^2 + 13 s[1] + 3)
         + (57 s[1]^3 + 237 s[1]^2 + 202 s[1] + 42)/(3 s[1]^3 + 13 s[1]^2 + 13 s[1] + 3) t
         - (8 s[1] + 7)/(s[1] + 1) t^2 + t^3

P_4(t) = 24 - 50 t + 35 t^2 - 10 t^3 + t^4
\end{lstlisting}

Finally, we specialise the free parameter by setting \(s_1=3\), that is, we run
\lstinline[style=wlcode]|OPRLFamily[{1, 2, 3, 4}, {3/2, 7/2}, {s[1] -> 3}]|.

\lstset{style=wlcode}
\begin{lstlisting}
OPRLFamily[{1, 2, 3, 4}, {3/2, 7/2}, {s[1] -> 3}];

OPRLFamily[{1, 2, 3, 4}, {3/2, 7/2}, {s[1] -> 3}];

Bands I_r = {{1},{2,3},{4}}

J_+ = {{1,2,4},{1,3,4}}

w(params) = {4/15+(2 s[1])/15,2/3,(2 s[1])/3,2/15+(4 s[1])/15}

specialisation = {s[1]->3}

w = {2/3,2/3,2,14/15}

OPRL P_k(t), k = 1,...,4:

P_1(t) = -(11/4)+t

P_2(t) = 21/4-5 t+t^2

P_3(t) = -(187/16)+18 t-(31 t^2)/4+t^3

P_4(t) = 24-50 t+35 t^2-10 t^3+t^4
\end{lstlisting}

Observe that the cubic polynomial
\[
P_3(t)=t^3-\frac{31}{4}t^2+18t-\frac{187}{16}
\]
in the latter output differs from
\[
P_3(t)=t^3-\frac{15}{2}t^2+\frac{269}{16}t-\frac{345}{32}
\]
in the former. In the present framework this is precisely what one should expect in the
{underdetermined} regime \(m<n-1\): fixing the two spectra \(Z_n\) and \(Z_m\) pins down
\(P_n\) and \(P_m\), but it leaves a genuine family of compatible Jacobi data, parametrised by the
strictly positive solution set \(\mathcal W_{>0}(Z_n,Z_m)\) of the Vandermonde system
\eqref{eq:vandermonde} (Lemma~\ref{lemma:OPRLlemma}). In the concrete instance of Example~\ref{ex:oprl}, strict interlacing yields the band decomposition
\[
I_0=\{1\},\quad I_1=\{2,3\},\quad I_2=\{4\},
\]
and Lemma~\ref{lem:positive-cone} shows that the nonnegative cone \(\mathcal W_{\ge 0}(Z_n,Z_m)\) is
generated by the circuit vectors supported on \(\mathfrak J_{+}\), i.e.\ by choosing exactly one
index from each band. Here \(\mathfrak J_{+}=\{\{1,2,4\},\{1,3,4\}\}\), so the cone has exactly two
extreme rays, generated by \(\boldsymbol\omega^{(1)}\) and \(\boldsymbol\omega^{(2)}\), and every
strictly positive solution is of the form
\[
\boldsymbol\omega=\boldsymbol\omega^{(1)}+s_1\,\boldsymbol\omega^{(2)},
\quad s_1\in I=(0,\infty),
\]
which is nothing but the relative interior description
$$
\mathcal W_{>0}(Z_n,Z_m)=\operatorname{relint}(\mathcal W_{\ge0}(Z_n,Z_m))
$$ in this one-parameter case.
Thus the ``admissible interval'' \(I\) for the free parameter is exactly the positivity window:
at the boundary \(s_1=0\) one weight vanishes and we leave \(\mathcal W_{>0}\), while for every
\(s_1>0\) all weights are strictly positive and the reconstruction mechanism applies. The two displayed cubics correspond to two different admissible choices of \(s_1\) inside \(I\):
the default output in the code is obtained at \(s_1=1\), whereas the specialised run
\(s_1=3\) produces the alternative \(P_1\) and \(P_3\) shown later. In other words, varying the
single parameter \(s_1\) within the admissible interval \(I\) moves one along the (non-unique)
solution set of the underdetermined two-spectra problem, changing the intermediate polynomials
(such as \(P_1\) and \(P_3\)) while keeping \(P_2\) and \(P_4\) fixed by the prescribed data.

The output of the following example is omitted, as it would require more printed pages than the present paper. The reader may, however, paste the code into \textsc{Mathematica}, evaluate it, and obtain the complete output.

\lstset{style=wlcode}
\begin{lstlisting}

OPRLFamily[{1, 3, 5, 7, 9, 11, 13, 15, 17, 19, 21, 23, 25, 27, 29, 31,
    33, 35, 37, 39, 41, 43, 45, 47, 49, 51, 53, 55, 57, 59, 61, 63, 
   65, 67, 69, 71, 73, 75, 77, 79, 81, 83, 85, 87, 89, 91, 93, 95, 97,
    99, 101, 103, 105, 107, 109, 111, 113, 115, 117, 119}, {2, 4, 6, 
   8, 10, 12, 14, 16, 18, 20, 22, 24, 26, 28, 30, 32, 34, 36, 38, 40, 
   42, 44, 46, 48, 50, 52, 54, 56, 58, 60, 62, 64, 66, 68, 70}];
\end{lstlisting}

\end{document}